\documentclass{amsart}
\usepackage{graphicx, amsmath, amssymb, amsthm, amsfonts, fullpage, latexsym,enumerate, xcolor, commath, tikz, tikz-cd, verbatim, hyperref,standalone, mathtools, booktabs}
\usepackage[hyphenbreaks]{breakurl}
\tikzcdset{scale cd/.style={every label/.append style={scale=#1},
    cells={nodes={scale=#1}}}}
\usepackage[backend=biber,style=alphabetic]{biblatex}
\usetikzlibrary{decorations.pathmorphing}

\title{On the topology of conormal complexes and posets of matroids}

\author{Anastasia Nathanson}
\address{School of Mathematics, University of Minnesota, Minneapolis MN 55455}
\email{natha129@umn.edu}

\author{Ethan Partida}
\address{Department of Mathematics, Brown University, Box 1917, Providence, RI 02912}
\email{ethan\_partida@brown.edu}


\theoremstyle{definition}
\newtheorem{defn}{Definition}[section]
\newtheorem{example}[defn]{Example}
\newtheorem{remark}[defn]{Remark}

\theoremstyle{plain}

\newtheorem{question}[defn]{Question}
\newtheorem{lem}[defn]{Lemma}
\newtheorem{prop}[defn]{Proposition}
\newtheorem{cor}[defn]{Corollary}

\theoremstyle{plain}
\newtheorem{thm}{Theorem}
\newtheorem{thmSpecial}{Theorem}
\newtheorem{corIntro}[thm]{Corollary}

\newcommand{\R}{\mathbb{R}} 
\newcommand{\Z}{\mathbb{Z}}

\newcommand{\FF}{\mathbf{F}} 
\newcommand{\FFd}{\mathbf{F^\perp}} 

\renewcommand{\L}{\mathcal{L}}

\newcommand{\F}{\mathcal{F}} 
\newcommand{\G}{\mathcal{G}} 
\renewcommand{\H}{\mathcal{H}} 
\newcommand{\K}{\mathcal{K}} 
\renewcommand{\L}{\mathcal{L}} 
\newcommand{\M}{\mathcal{M}}

\newcommand{\FG}{F \vert G}

\newcommand{\cl}{\mathrm{cl}}

\newcommand{\BF}{\mathbf{BF}} 
\newcommand{\BS}{\mathbf{BS}}

\newcommand{\op}{\mathrm{op}}

\newcommand{\diagarrow}{\,\rotatebox[origin=c]{-45}{$\rightarrow$}\,}

\renewcommand{\phi}{\varphi}

\DeclareMathOperator\trop{\mathrm{trop}}
\DeclareMathOperator\rk{\mathrm{rk}}
\DeclareMathOperator\cork{\mathrm{cork}}
\DeclareMathOperator\Aut{\mathrm{Aut}}

\DeclareMathOperator\cone{\mathrm{cone}}
\DeclareMathOperator\lk{\mathrm{lk}}
\DeclareMathOperator\dell{\mathrm{del}}
\DeclareMathOperator\str{\mathrm{str}}

\begin{document}

\begin{abstract}
We introduce the poset of biflats of a matroid $M$, a Lagrangian analog of the lattice of flats of $M$, and study the topology of its order complex, which we call the biflats complex. This work continues the study of the Lagrangian combinatorics of matroids, which was recently initiated by work of Ardila, Denham and Huh. We show the biflats complex contains two distinguished subcomplexes: the conormal complex of $M$ and the simplicial join of the Bergman complexes of $M$ and $M^\perp$, the matroidal dual of $M$. Our main theorems give sequences of elementary collapses of the biflats complex onto the conormal complex and the join of the Bergman complexes of $M$ and $M^\perp$. These collapses give a combinatorial proof that the biflats complex, conormal complex and the join of the Bergman complexes of $M$ and $M^\perp$ are all simple homotopy equivalent. Although simple homotopy equivalent, these complexes have many different combinatorial properties. We collect and prove a list of such properties.
\end{abstract}

\maketitle

\section{Introduction}
The Bergman fan of a matroid $M$, first defined by Ardila and Klivans \cite{AK}, is a simplicial fan with support equal to the tropical linear space associated to the matroid, $\trop(M)$. The work of Ardila, Denham, and Huh \cite{ADH} introduces a Lagrangian analog of the well-studied Bergman fan. This fan, called the \emph{conormal fan} of the matroid, is a simplicial fan with support equal to the product $\trop(M)\times \trop(M^\perp)$, where $M^\perp$ is the dual matroid of $M$. For motivation behind the descriptor ``Lagrangian'' see \cite[Sec. 1.2]{ADH}.

The conormal fan 
provides an intersection-theoretic interpretation of the $h$-vector of the broken circuit complex of $M$. 
Ardila, Denham and Huh used this interpretation to affirm long-standing conjectures on the log concavity of this $h$-vector posed in \cite{dawson}, \cite{swartz}, and \cite{brylawski}. 
The combinatorics of the conormal fan was further explored by Ardila, Denham and Huh in \cite{ADHCombo}. 
There, they give a combinatorial proof that certain intersection numbers in the Chow ring of the conormal fan $M$ are the $h$-vector of the broken circuit complex of $M$.

We continue the exploration of the Lagrangian combinatorics of matroids by studying the simplicial complex associated to the conormal fan, the \emph{conormal complex}.
Unlike the Bergman complex, $\Delta_M$, the conormal complex, denoted $\Delta_{M, M^\perp}$, is combinatorially unruly. We show that the conormal complex is generally \emph{not flag} and hence cannot be the order complex of a poset. 

We study the relationship between the conormal complex $\Delta_{M, M^\perp}$ and the join of the Bergman complexes $\Delta_M *\Delta_{M^\perp}$ by including both complexes as subcomplexes of a larger simplicial complex. We call this larger complex the \emph{biflats complex} and denote it by $\Delta(\BF_{M,M^\perp})$.
Just as the Bergman complex is the order complex of the lattice of flats of a matroid, the biflats complex is the order complex of the \emph{poset of biflats}, denoted $\BF_{M, M^\perp}$. A biflat is the Lagrangian analogue of a flat, introduced in \cite{ADH}. Each biflat, denoted $F\vert G$, consists of a flat $F$ of $M$ and a flat $G$ of $M^\perp$ such that their union is the ground set $E$; see Section \ref{sec:biflats} for a more detailed definition.
The set of biflats is ordered by inclusion in the first component and reverse inclusion in the second component. The poset $\BF_{M,M^\perp}$ is still a combinatorially difficult object. It is often the case that there is a minimal biflat $F\vert G$ and maximal biflat $F'\vert G'$ such that two maximal chains between $F\vert G$ and $F'\vert G'$ have different lengths. This means that $\BF_{M,M^\perp}$ is not a ranked poset and thus, the biflats complex $\Delta(\BF_{M,M^\perp})$ is not pure dimensional. Furthermore, the biflats complex is not (non-pure) shellable. 
 
The conormal complex can be defined as a subcomplex of the biflats complex and hence naturally includes as a subcomplex. In contrast, the product of Bergman complexes $\Delta_M*\Delta_{M^\perp}$  is not defined as a subcomplex of the biflats complex. We show that $\Delta_M * \Delta_{M^\perp}$ is indeed isomorphic to a sub order-complex of $\Delta(\BF_{M,M^\perp})$ which we call the \emph{unmixed biflats complex} and denote by $\Delta(\BF_{M,M^\perp}^u)$.

Ardila, Denham and Huh prove that the conormal complex $\Delta_{M,M^\perp}$ and the join of the Bergman complexes $\Delta_M* \Delta_{M^\perp}$ are related by a sequence of edge subdivisions and inverse edge subdivisions \cite[Theorem 6.25]{ADH}. These operations preserve homotopy type. In fact, these operations preserve a stronger equivalence, namely, \emph{simple homotopy type}. Two simplicial complexes are simple homotopy equivalent if there exists a third simplicial complex which \emph{collapses} onto both of them. A collapsing of a simplicial complex $\Delta$ onto a subcomplex $\tau$ is a sequence of simplicial complexes 
\[\Delta=\Delta_0\diagarrow \Delta_1\diagarrow \cdots\diagarrow \Delta_k=\Delta'\]
such that, for all $i$, $\Delta_i$ is obtained from $\Delta_{i-1}$ by an \emph{elementary collapse}. An elementary collapse $\Delta_{i-1}\diagarrow \Delta_i$ is a simple combinatorial operation which deletes some faces from $\Delta_{i-1}$ to obtain the subcomplex $\Delta_i$. These deletions are defined so that $\Delta_{i-1}$ and $\Delta_i$ are guaranteed to have the same homotopy type. For more details on these concepts, see Section \ref{sec:simpleHomotopy}.

Our main contribution is to make explicit the simple homotopy equivalence between the conormal complex $\Delta_{M,M^\perp}$ and the join of Bergman complexes $\Delta_M*\Delta_{M^\perp}$. 
That is, we give explicit collapses of $\Delta(\BF_{M,M^\perp})$ onto both $\Delta_{M,M^\perp}$ and $\Delta_M*\Delta_{M^\perp}$. These collapses are defined purely matroidally and rely heavily on an understanding of the combinatorics of biflats.

\begin{thm}
\label{thm:unmixed}
    Given a matroid $M$, there is an explicit sequence of elementary collapses
    \[\Delta(\BF_{M,M^\perp})\diagarrow \Delta_1 \diagarrow \cdots \diagarrow \Delta(\BF_{M,M^\perp}^u)\cong \Delta_M * \Delta_{M^\perp} \]
    from the biflats complex $\Delta(\BF_{M,M^\perp})$ onto the join of the Bergman complexes $\Delta_M * \Delta_{M^\perp}$. 
\end{thm}
\begin{thm} \label{thm:conormal}
    Given a matroid $M$, there is an explicit sequence of elementary collapses
    \[\Delta(\BF_{M,M^\perp})\diagarrow \Delta_1 \diagarrow \cdots \diagarrow  \Delta_{M,M^\perp}\]
    from the biflats complex $\Delta(\BF_{M,M^\perp})$ onto the conormal complex $\Delta_{M,M^\perp}$. 
\end{thm}

From these two theorems, we recover the simple homotopy equivalence between $\Delta_{M,M^\perp}$ and $\Delta_M* \Delta_{M^\perp}$ guaranteed by \cite[Theorem 6.25]{ADH}.
\begin{corIntro}
    Given a matroid $M$, the conormal complex $\Delta_{M,M^\perp}$ is simple homotopy equivalent to the join of Bergman complexes $\Delta_M*\Delta_{M^\perp}$.
\end{corIntro}

When $M$ is a uniform matroid, we show that the conormal complex, biflats complex and unmixed biflats complex are all equal as simplicial complexes in Lemma \ref{lem:unmixed_join}. 
However, if $M$ is not a uniform matroid, these complexes, while all homotopic, have different topological and combinatorial properties. We collect a table of the major properties in Table ~\ref{tab:table}.

\begin{table}[ht] 
\label{tab:table}
\centering
    \begin{tabular}{lccc}\toprule
         {Properties} &  {Conormal Complex} &  {Biflats complex} & {Unmixed Biflats complex} \\ 
         {} &  {$\Delta_{M,M^\perp}$} &  {$\Delta\left(\BF_{M,M^\perp}\right)$} & {$\Delta\left(\BF_{M,M^\perp}^{u}\right)$} \\\midrule
         Dimension & {\large$|E|-3$}&  \large{$|E|-2$}& \large{$|E|-3$}\\\addlinespace
         Pure & { \huge\checkmark} & \huge$\times$ & \huge\checkmark\\\addlinespace
         Flag & \huge$\times$ & \huge\checkmark & \huge\checkmark\\\addlinespace
         (Non-pure) Shellable & \huge? & \huge$\times$ & \huge\checkmark \\\addlinespace
         Homeomorphic to $\Delta_M*\Delta_{M^\perp}$ & \huge\checkmark & \huge$\times$ & \huge\checkmark  \\\addlinespace 
         \midrule
    \end{tabular}
    \caption{Properties of the conormal complex, biflats complex and unmixed biflats complex for a non-uniform matroid $M$. References for the claims, from left to right, are as follows: 
    Dimension (\cite[Lemma 2.1]{ADHCombo}, Lemma \ref{lem:biflats_dim}, Lemma \ref{lem:unmixed_join}), 
    Pure (\cite[Lemma 2.1]{ADHCombo}, Example \ref{ex:small_triangle}, Lemma\ref{lem:unmixed_join}), 
    Flag (Example \ref{ex:big_triangle}, Definition \ref{def:biflat_complex}, Lemma \ref{lem:unmixed_join}), 
    Shellable (Remark \ref{rmk:shellable}, Example \ref{ex:big_triangle}, Lemma \ref{lem:unmixed_join}), 
    Homeomorphic (Proposition \ref{prop:homeomorphic}, Lemma \ref{lem:biflats_dim}, Lemma\ref{lem:unmixed_join}).
      }
\end{table}

We provide the necessary background definitions in Sections ~\ref{sec:simp} through ~\ref{sec:biflats} with a look at the fan structure behind our definitions in Section ~\ref{sec:conormal_back}. 
We introduce the objects at play in Theorem ~\ref{thm:unmixed} in Section ~\ref{sec:unmixedbackground} and prove the theorem in Section ~\ref{sec:unmixed}. We introduce the objects of Theorem ~\ref{thm:conormal} in Section ~\ref{sec:conormalbackgroundsection} and prove the theorem in Section ~\ref{sec:conormalcomplex}. 
We collect examples of Lagrangian objects in Section ~\ref{sec:examples}. Finally, Section~\ref{sec:future} collects some further questions related to this work.

\section*{Acknowledgments}
The authors are grateful to Victor Reiner for suggesting this project and for continued guidance during all of its stages, as well as Caroline Klivans and Melody Chan for many insightful suggestions and questions. The authors would also like to thank Federico Ardila, Graham Denham, Trevor Karn, Trent Lucas, Gregg Musiker and Raul Penaguiao for helpful references and conversations. They thank the anonymous referees for helpful suggestions that improved the exposition. The first author was partially supported by NSF DSM-2053288.

\section{Background} \label{sec:background}

In Sections \ref{sec:simp} through \ref{sec:conormalbackgroundsection}, we introduce all of our necessary definitions. Section \ref{sec:examples} provides three worked out examples of these definitions. The examples introduced in \ref{sec:examples} will be referenced throughout the following sections of the article

\subsection{Posets and simplicial complexes} \label{sec:simp}

Here, we collect a few definitions related to simplicial complexes and posets as we prepare to discuss the central objects of our work.


A \emph{simplicial complex} $\Delta$ on vertex set $V$ is a collection of subsets of $V$ such that if $\sigma \in \Delta$ and $\tau \subset \sigma$ then $\tau\in \Delta$. We refer to the sets $\sigma \in \Delta$ as the \emph{faces} of $\Delta$. A \emph{facet} of $\Delta$ is an inclusion-maximal face of $\Delta$. The \emph{dimension} of a face is equal to one less than its cardinality. The dimension of $\Delta$ is the maximum dimension of its faces. We say that $\Delta$ is \emph{pure} if all of its facets are the same dimension. 

A poset $P$ is \emph{bounded} if there are elements $\hat 0,\hat 1 \in P$ such that for all $x\in P$, $\hat 0 \leq x \leq \hat 1$. A bounded poset $P$ is \emph{ranked} if for all $x\in P$, all saturated chains from $\hat 0$ to $x$ are the same length. We often represent a poset $P$ by its \emph{Hasse diagram}. The Hasse diagram of a poset is the directed graph whose vertices are the elements of $P$ and whose edges are the pairs $(x,y)$ such that $x<y$. We draw our Hasse diagrams as undirected graphs with the implicit understanding that if there is an edge from $x$ to $y$ and $y$ is above $x$, then $x< y$. The {order complex} $\Delta(P)$ of $P$, is the simplicial complex whose vertex set is the elements of $P$ and whose faces are the chains in $P$. 
A simplicial complex  is \emph{flag} if all of its inclusion-minimal non-faces are two element sets. 
Since partial orders are transitive, every order complex is flag.
For more detailed discussions, we recommend  \cite{Bjorner} and \cite{EC1} as references for simplicial complexes and posets, respectively.

For a face $F$ of a simplicial complex $\Delta$, 
\begin{itemize}
    \item the \emph{link} of $F$ is the simplicial complex $\lk_\Delta(F) = \{G\in \Delta: G\cap F = \emptyset, G\cup F \in \Delta\};$
    \item the \emph{open star} of $F$  is the set of faces, $\str_\Delta(F)= \{G\in \Delta: F\subseteq G \};$
    \item the \emph{deletion} of face $F$ of $\Delta$ is the simplicial complex 
    $\dell_\Delta(F)= \{G\in \Delta: F\not\subseteq G\}.$
\end{itemize}

\noindent 
Note that the deletion of face $F$ can be equivalently defined as $\dell_\Delta(F)\coloneqq \Delta \setminus \str_\Delta(F).$ The deletion of faces is commutative; that is, for faces $F,G$ in a simplicial complex $\Delta$,
\[\dell_{\dell_{\Delta}(F)}(G)=\dell_{\dell_{\Delta}(G)}(F) = \{H\in \Delta: F\not\subseteq H, G\not\subseteq H\}.\]
Additionally, 
    the \emph{join} of two disjoint simplicial complexes $\Delta_1 $ and $\Delta_2$ is the complex 
    \[\Delta_1 * \Delta_2 \coloneqq \{\sigma \cup \tau : \sigma\in \Delta_1, \tau\in \Delta_2\}.\]

\noindent
The \emph{ordinal sum} of posets  $X\oplus Y$ is the poset on the disjoint union of elements in $X, Y$ such that $p\leq q$ either if $p \leq q$ as elements of $X$, elements of $Y$, or, if $p\in X$ and $q\in Y$.

\begin{lem} \label{lem:join}
    If $P$ is the ordinal sum of posets $X \oplus Y$, then its order complex $\Delta(P)$ is the join of $\Delta(X)*\Delta(Y)$.
\end{lem}
\noindent Lemma \ref{lem:join} follows immediately from the definition of ordinal sums and order complexes of posets.

\begin{defn}
A simplicial complex $\Delta$ contains a \emph{cone vertex} $v$ if $\Delta$ can be written as $v*\Delta'$ where $\Delta'$ is a subcomplex of $\Delta$. This is equivalent to the condition that all inclusion-maximal faces of $\Delta$ contain $v$.
\end{defn}

\subsection{Simple homotopy equivalence} \label{sec:simpleHomotopy}
As explained in the introduction, the aim of this paper is to explicitly show that $\Delta_{M,M^\perp}$ is simple homotopy equivalent to $\Delta_M * \Delta_{M^\perp}$. We will do this by constructing a sequence of elementary collapses from the biflats complex $\Delta(\BF_{M,M^\perp})$ onto both $\Delta_{M,M^\perp}$ and $\Delta_M * \Delta_{M^\perp}$. We review the definition of an elementary collapse below.

\begin{defn}
Let $\sigma$ be a $k$-dimensional inclusion-maximal face of a simplicial complex $\Delta$ and $\tau$ be a $k-1$-dimensional face of $\Delta$. Suppose that $\tau\subset \sigma$ and $\tau$ is not contained in any other inclusion-maximal face of $\Delta$. Then $\tau$ is a \emph{free face} of $\sigma$ and we say that $\Delta$ has an \emph{elementary collapse} onto the subcomplex $\Delta' = \Delta\setminus \{\sigma, \tau \}$ and denote this by $\Delta \diagarrow \Delta'$. Elementary collapses preserve homotopy type. 
\end{defn}

\begin{defn}
    If there is a sequence of elementary collapses from a complex $\Delta$ onto a subcomplex $\Delta'$, we say that $\Delta$ \emph{collapses} onto $\Delta'$. We say that two complexes $\Delta_1$ and $\Delta_2$ are \emph{simple homotopy equivalent} if there is a larger complex $\Delta_3$ which collapses onto both $\Delta_1$ and $\Delta_2$.
\end{defn}


Simple homotopy equivalence is a combinatorial refinement of homotopy equivalence; there exists homotopy equivalent complexes which are not simply homotopy equivalent. For examples of this phenomenon and a broader introduction to simple homotopy theory, see \cite{Cohen}. 
We make use of the following lemmas on elementary collapses.

\begin{lem}\cite[Lemma 2.3]{welker}
\label{lem:cone_collapse}
Any simplicial complex $\Delta$ which has a cone vertex $v$ is collapsible to $v$.
\end{lem}

\begin{lem} \cite[Lemma 2.7]{welker}
\label{lem:collapse}
If a simplicial complex $\Delta$ has a face $F$ for which $\lk_\Delta(F)$ is collapsible to a vertex $v$ in $\lk_\Delta(F)$, then $\Delta$ is collapsible to $\dell_\Delta(F)$. 
\end{lem}

\noindent For any poset $P$, we use $P_{<x}$ to denote the subposet that contains all elements $y\in P$ such that $y<x$ in $P$.

\begin{lem}\label{lem:pos_collapse}
If $P$ is a poset and $x\in P$ such that $P_{<x}$ has a \emph{unique} maximal element then there exists a sequence of elementary collapses from $\Delta(P)$ onto $\Delta(P\setminus\{x\})$. 
\end{lem}
\begin{proof}
Observe that $\lk_{\Delta(P)}(x)= \Delta\left(P_{>x}\right)*\Delta\left(P_{<x}\right)$. This is because we can decompose any chain $C$ of $P$ containing $x$ as
\[C = C_{<x} <  x < C_{>x} \]
where $C_{<x}$ and $C_{>x}$ are chains of elements less than $x$ and greater than $x$, respectively. If $P_{<x}$ has a maximal element, then its order complex $\Delta(P_{<x})$ has a cone vertex and $\Delta\left(P_{>x}\right)*\Delta\left(P_{<x}\right)$ will have a cone vertex as well. By Lemma ~\ref{lem:collapse}, this means there exists a sequence of elementary collapses from $\Delta(P)$ onto $\dell_{\Delta(P)}(x) = \Delta(P\setminus\{x\})$. \qedhere
\end{proof}

\subsection{Matroids and their flats}
\label{sec:matroids}
Broadly, we study order complexes defined by the data of a matroid $M$. We refer the reader to \cite{oxley} for the basics of matroid theory and Sections $5$ and $6$ of \cite{icm_matroids} for a friendly introduction to the geometry of Bergman and conormal fans.

We will think of a \emph{matroid} as a tuple $M = (E, \rk_M)$ given by a finite \emph{ground set} $E$ and a submodular function called \emph{rank} on the subsets of the ground set, $\rk_M: 2^E \to \Z$ such that $\rk_M(\emptyset) = 0$ and, for any set $S\subseteq E$ and element $x\in E$,  
\[\rk_M(S) \leq \rk_M(S\cup \{x\}) \leq \rk_M(S)+1 .\] 
We forego the subscript when the matroid is clear from context. The rank of a matroid $M$ is the number $\rk_M(E)$.

For an arbitrary subset $S\subseteq E$, its \emph{closure} is 
$\cl_M(S)\coloneqq\{x\in S \colon \rk_M(S) =\rk_M\left(S\cup\{x\}\right)\}$. Sets that are their own closure are called \emph{flats}. The collection of flats, $\FF$, when ordered by inclusion forms a geometric lattice; in an abuse of notation, we will use $\FF$ to refer to both the lattice and the set of flats. 

A flat is \emph{proper} if it is not the ground set $E$. An element $x\in E$ is a loop of $M$ if it is contained in $\cl_M(\emptyset)$.
Define the \emph{corank} of a flat $F$ to be equal to $\cork_M(F) \coloneqq \rk_M(E)-\rk_M(F)$. A set $C\subseteq E$ is a \emph{circuit} if, for every $c\in C$, we have that $\rk_M(C) = \rk_M(C\setminus\{c\})=|C|-1$. 

The dual matroid, $M^\perp = (E, \rk_{M^\perp}$), is a matroid defined on the same ground set as $M$ with the rank function $\rk_{M^\perp}(S) = |S|-\rk_M(S)+\cork_M(S)$. The rank of a matroid and the rank of its dual matroid dual sum to the size of the ground set: $\rk_M(E) + \rk_{M^\perp}(E) = \abs{E}.$

The flats of the dual matroid, which we will call \emph{dual flats}\footnote{Note, flats of the dual matroid are also sometimes called ``coflats'', e.g. in \cite{ADH} and \cite{ADHCombo}. To avoid overloading the prefix ``co-'', we use the terminology dual flats. }. 
can be characterized as the complements of unions of circuits of $M$. We denote the collection of dual flats as $\FFd$. 

An element $i\in E$ is a \emph{coloop} of $M$ if it is a loop of $M^\perp$. We require all matroids to be loopless and coloopless. We provide a reasoning for this requirement in Remark ~\ref{rem:loops} after some of the discussion in Section ~\ref{sec:conormal_back}.

\emph{Uniform matroids} form an important and simple class of matroids. A matroid $M$ of rank $r$ on a ground set of size $n$, is a uniform matroid if $\rk_M(S) = \min\{r, |S|\}$  for every $S\subseteq E$; here, necessarily $r\leq n$. 
It follows that the circuits of a uniform matroid are all the cardinality $r+1$ subsets of $E$.
Up to a relabeling of the ground set, all rank $r$ uniform matroids on a ground set of size $n$ are equal. We denote such a matroid with the notation $U_{r,n}$. 
The matroid dual to $U_{r,n}$ is the uniform matroid $U_{n-r,n}$. We prove another characterization of uniform that will be useful throughout the rest of the paper.


\begin{prop}\label{prop:uniform_flats}
    A matroid $M$ is a uniform matroid if and only if there does not exist a proper flat $F$ and proper dual flat $G$ such that $F\cup G =E$.
\end{prop}

 \begin{proof}
 We can check this property by only considering inclusion maximal proper flats $F$ and inclusion maximal proper dual flats $G$. 
 By \cite[Proposition 2.1.6]{oxley}, we can write $F=E\setminus C^\perp$ and $G=E\setminus C$ where $C$ and $C^\perp$ are a circuit of $M$ and a circuit of $M^\perp$, respectively. Thus our claim is true if and only if for every circuit $C$ of $M$ and every circuit $C^\perp$ of $M^\perp$,
\[\left(E\setminus C^\perp\right)\cup \left(E\setminus C\right) \neq E \iff C^\perp\cap C \neq \emptyset. \]
It is a well known exercise that a matroid $M$ is uniform if and only if every circuit of $M$ intersects every circuit of $M^\perp$. See, for example, \cite[Exercise 3.24]{OxleySurvey}.
 \end{proof}

Lastly, we define a central object of this article.

\begin{defn}
    Given a matroid $M = (E, \rk_M)$, its \emph{Bergman complex}, $\Delta_M$, is the order complex of the proper, non-empty flats, $\FF\setminus\{\emptyset, E\}$, of the matroid. 
\end{defn}

\subsection{Lagrangian analogue to flats}\label{sec:biflats} 
The Langiangian theory of matroids was started by the work of 
Ardila, Denham, and Huh in \cite{ADH} and further continued in \cite{ADHCombo}. Most important to us is the Lagrangian analogue of a flat.

\begin{defn}
Given a loopless and coloopless matroid $M$, 
a \emph{biflat} $F \vert G$ consists of a flat $F\in \FF$ and a dual flat $G\in \FFd$ such that both are nonempty, at least one of them is proper and 
$F \cup G = E.$
We denote the set of all biflats as $\BF_{M,M^\perp}$, dropping the subscript when the matroid is clear. 
\end{defn}
\noindent Note that the last condition in the definition above can be recharacterized by saying that $\FG$ is a biflat if the closure of the complement of the flat $F$ in $M^\perp$ is contained within $G$, or $\cl_{M^\perp}(F^C) \subseteq G$. 

Just as flats can be ordered by inclusion, biflats can be ordered by inclusion conditions.
\begin{defn} \label{biflat_order}
  The \emph{poset of biflats} of a matroid $M$, denoted $\BF_{M,M^\perp}$, is the partial ordering of the set of biflats of $M$ by the condition that
  \[F\vert G \leq F'\vert G' \text{ if and only if } F \subseteq F' \text{ and } G\supseteq G'. \]
 We call a chain $\F \vert \G$ in $\BF_{M,M^\perp}$ a \emph{bichain}. 
\end{defn}

 Under this order, every biflat is preceded by $\emptyset \vert E$ and precedes
 $E\vert \emptyset$. While these are not biflats themselves, these ``pseudo-biflats'' can always be attached at the beginning and end of bichains. 
 That is, for any bichain $\F\vert \G =\{F_1\vert G_1 \leq \ldots \leq F_{\ell}\vert G_{\ell}\},$ it is customary to assume that $F_0\vert G_0 = \emptyset \vert E$ and $F_{\ell+1}\vert G_{\ell+1}= E\vert \emptyset$.

\begin{defn} \label{def:biflat_complex}
   The \emph{biflats complex} $\Delta\left(\BF_{M,M^\perp}\right)$ is the order complex of the poset of biflats. 
\end{defn}

We now record the dimension of the biflats complex. The proof we give is an adaptation of the proof of \cite[Lemma 2.1]{ADHCombo}.

\begin{lem}
  \label{lem:biflats_dim}
    The biflats complex of a non-uniform matroid $M$ has dimension $|E|-2$.
\end{lem}
\begin{proof}
Let $\rk_M(E) \coloneqq r$ and $\rk_{M^\perp}(E) \coloneqq r^\perp$. 
Suppose that $\mathcal{F}\vert\mathcal{G}=\{F_1\vert G_1 \leq \ldots \leq F_\ell\vert G_\ell\}$ is a bichain of $\BF_{M,M^\perp}$. 
Setting $F_0\vert G_0= \emptyset\vert E$ and $F_{\ell+1}\vert G_{\ell+1} = E\vert \emptyset$, we can see
\[\sum_{j=0}^\ell \left(\rk_M(F_{j+1})-\rk_M(F_j)\right) + \left(\rk_{M^\perp}(G_{j})-\rk_{M^\perp}(G_{j+1})\right)  = \rk_M(E)+\rk_{M^\perp}(E) = r+ r^\perp= \abs{E}.\]
Each term in this sum is positive, so we must have $\ell\leq \abs{E}-1$. To prove this bound is tight, we exhibit a bichain with length $\abs{E}-1$. Since $M$ is not a uniform matroid, by Proposition \ref{prop:uniform_flats}, 
we can assume there exists some 
corank $1$ flat $F_{r-1}$ and corank $1$ dual flat $G_{r^\perp-1}$ such that $F_{r-1}\cup G_{r^\perp-1} = E$. 
Let 
\[F_1\subseteq F_2 \subseteq \ldots \subseteq F_{r-2}\subseteq F_{r-1} \quad \text{and} \quad G_1\subseteq G_2 \subseteq \ldots \subseteq G_{r^\perp-2} \subseteq G_{r^\perp-1} \]
be maximal increasing chains of proper, non-empty flats and dual flats respectively. Then 
\[F_1\vert E  \leq \ldots \leq F_{r-1}\vert E \leq F_{r-1}\vert G_{r^\perp-1} \leq E \vert G_{r^\perp-1}\leq \ldots \leq E\vert G_1\]
is a bichain of length $\abs{E}-1$.
\end{proof}

\subsection{Unmixed biflats complex}\label{sec:unmixedbackground}
We discuss the \emph{unmixed} biflats of a matroid. The unmixed biflats are a type of biflat first defined in
 Section 4.3 of \cite{ADHCombo}. We use the subposet formed by them to define the \emph{unmixed biflats complex}. 

\begin{defn}
     A biflat $F\vert G$ is \emph{mixed} if both $F$ is a proper flat of $M$ and $G$ is a proper flat of $M^\perp$. Otherwise, we say the biflat $F\vert G$ is \emph{unmixed}. 
     We denote the set of unmixed biflats as 
     \[\BF_{M,M^\perp}^{u} \coloneqq \{F\vert G  \colon F=E \text{ or } G= E\} \subseteq \BF_{M,M^\perp}.\]
\end{defn}

 \noindent As $\BF_{M,M^\perp}^{u}$ is a subset of $\BF_{M,M^\perp}$, the set of unmixed biflats inherits a poset structure. Once again, the set and the poset share the same notation. We call the order complex of $\BF_{M,M^\perp}^{u}$ the \emph{unmixed biflats complex} and denote it by $\Delta(\BF_{M,M^\perp}^u)$.

Because biflats are ordered with respect to inverse inclusion on dual flats, the poset of unmixed biflats is isomorphic to the ordinal sum of $\FF$ and 
$\left(\FFd\right)^{\op}$. 
For any poset $P$, the order complex $\Delta(P^\op) = \Delta (P) $, so the following lemma should be unsurprising.

\begin{lem}
  \label{lem:unmixed_join}
    The order complex of unmixed biflats is isomorphic to the join of $\Delta_M$ and $\Delta_{M^\perp}$. 
\end{lem}
\begin{proof}
Consider the map $\phi:    \BF^u_{M, M^\perp} \to \FF \oplus (\FFd)^\op  $ given by 
\[\phi(F\vert G) =
\begin{cases}
    F &\text{if } G=E \\
    G &\text{if } F=E. 
\end{cases}
\]
One can check that $\phi$ is a poset isomorphism and hence, by Lemma ~\ref{lem:join}, \[\Delta(\BF_{M,M^\perp}^u)  \cong \Delta_M * \Delta_{M^\perp}.\qedhere\]
\end{proof}

With our terminology now in place, the following corollary is a rephrasing of Proposition \ref{prop:uniform_flats}.

\begin{cor}\label{cor:uniform_unmixed}
    A matroid $M$ has no mixed biflats if and only if $M$ is a uniform matroid.
\end{cor}

\subsection{Conormal complex}\label{sec:conormalbackgroundsection}
Lastly, we define the \emph{conormal complex} $\Delta_{M,M^\perp}$. This is the simplicial complex associated to the \emph{conormal fan}. We provide a more an in-depth study of the conormal fan in Section ~\ref{sec:conormal_back}. We again refer to \cite{ADH} and \cite{ADHCombo} as sources for our definitions. 

\begin{defn}\label{def:biflag}
A \emph{biflag}, $\F\vert\G$, of a matroid $M$ is a bichain satisfying the restriction that 
\[\bigcup_{F\vert G\in \F\vert\G}F\cap G \neq E. \]
The \emph{conormal complex} $\Delta_{M,M^\perp}$ of a matroid is the simplicial complex whose faces are biflags of $M$.
\end{defn}

\noindent It is shown in Proposition $2.15$ of \cite{ADH}, that a bichain is a biflag if and only if it satisfies the following \emph{gap condition}. A bichain \[\F\vert\G = \{F_1\vert G_1 \leq F_2\vert G_2\leq \ldots \leq F_\ell\vert G_\ell\}\]satisfies the gap condition if, when we let $F_0\vert G_0 = \emptyset\vert E$ and $F_{\ell+1}\vert G_{\ell+1}= E\vert\emptyset$, there exists an index $0\leq j\leq \ell$ such that $F_j\cup G_{j+1}\neq E$.

The requirement that faces of the conormal complex be biflags usually complicates its structure. In fact, as we will see in Example \ref{ex:big_triangle}, the conormal complex often fails to be a flag simplicial complex. However, if $M$ is a uniform matroid, then the conormal complex of $M$, biflats complex of $M$ and unmixed biflats complex of $M$ are all equal. In fact, these equalities are a characterization of uniform matroids. 




\begin{prop} \label{prop:tfae}
For a loopless and coloopless matroid $M$, the following are equivalent:
\begin{itemize}
\item $M$ is a uniform matroid
\item The biflats complex $\Delta\left(\BF_{M,M^\perp}\right)$ is the unmixed biflats complex $\Delta\left(\BF^u_{M,M^\perp}\right)$.
\item The biflats complex $\Delta\left(\BF_{M,M^\perp}\right)$ is the conormal complex $\Delta_{M,M^\perp}$.
\end{itemize}
\end{prop}
\begin{proof}
    The equivalence of the first and second bullet is Corollary \ref{cor:uniform_unmixed}. 
    To show the second and third bullet are equivalent, we will prove that any matroid
    $M$ has a mixed biflat if and only if there exists a bichain of $M$ which is not a biflag. 
    If $F\vert G$ is a mixed biflat of $M$, then $F\vert E \leq E \vert G$ is a bichain of $M$ which fails the gap condition and hence is not a biflag of $M$. Now suppose that $\F\vert \G = F_1\vert G_1 \leq F_2 \vert G_2 \leq \ldots \leq F_k\vert G_k$ is a bichain of $M$ which is not a biflag. If any of the biflats of $\F\vert \G$ are mixed, we are done. If not, then there exists an index $1\leq i< k$ such that $F_i\vert G_i = F\vert E$ and $F_{i+1}\vert G_{i+1} = E\vert G$ for some proper flat $F$ and proper dual flat $G$. But as $\F\vert \G$ is not a biflag, the gap condition tells us that $F\cup G =E$ and thus $F\vert G$ is a mixed biflat of $M$.
\end{proof}

\subsection{Examples} \label{sec:examples}
Throughout this paper, we will refer to \emph{graphic matroids} as examples. A graphic matroid $M$ is specified from a graph $G$ in the following manner. The ground set $E$ is given by the edge set of $G$. The rank of a subset of edges is the maximum size of an acyclic subset. 
Hence, the flats $\FF$ are the edge subsets $F \subseteq E$ which are transitively closed:  if $F$ contains a path of edges from vertex $v$ to vertex $v'$, and if $e=\{v,v'\} \in E$, then $e \in F$. If $G$ is a planar graph, then $M^\perp$ is the graphic matroid of the dual graph of $G$.

Our examples will gradually increase in complexity to highlight different features of the conormal complex, biflats complex and unmixed biflats complexes.
\begin{example}\label{ex:small_triangle}
Let $M$ be the matroid defined by the graph below; its dual graph, which defines the dual matroid, is also below.  The flats of $M$ are $\mathbf{F} = \{\emptyset, 12, 3, 4, E\}$ and the flats of the dual matroid are $\mathbf{F}^\perp = \{\emptyset, 1, 2, 34, E\}$, where we forego the set braces and commas when writing a flat. 
    \begin{center}
    \begin{tikzpicture} [scale=.8]
\node at (-1, 0.5) {\Large $M$};

\node[circle,draw] (a) at (0,0) {};
\node[circle,draw] (b) at (4,0) {};
\node[circle,draw] (c) at (2,2){};

\draw (a) to[bend right] node[below]{$1$} (b) ;
\draw (a) to[bend left] node[above]{$2$} (b) ;
\draw (a) to node[left]{$3$} (c) ;
\draw (b) to node[right]{$4$} (c) ;

\end{tikzpicture} \hspace{1in}
    \begin{tikzpicture} [scale=.8]
\node at (-1, 0.5) {\Large $M^\perp$};

\node[circle,draw] (a) at (0,0) {};
\node[circle,draw] (b) at (4,0) {};
\node[circle,draw] (c) at (2,2){};

\draw (a) to[bend right] node[below]{$3$} (b) ;
\draw (a) to[bend left] node[above]{$4$} (b) ;
\draw (a) to node[left]{$1$} (c) ;
\draw (b) to node[right]{$2$} (c) ;

\end{tikzpicture}
    \end{center}

The biflats of this matroid are
$\BF = \{ 12 \vert E, 12\vert 34, 3 \vert E, 4\vert E,  E\vert 1, E\vert 2, E\vert 34\}.
$
In the figure below, the Hasse diagram of $\BF_{M, M^\perp}$ is shown on the left and its order complex $\Delta(\BF_{M, M^\perp})$ is shown on the right. All elements of $\BF_{M, M^\perp}$ and all vertices of $\Delta(\BF_{M, M^\perp})$ are labeled. Any unlabeled intersections between edges are coincidental and do not correspond to poset elements or vertices.

\begin{center}
  \begin{tikzpicture} [scale=.35,auto=left]
 \node (12e) at (9,1) {$12\vert E$};
  \node (3e) at (3,1) {$3\vert E$};
  \node (4e) at (-3,1) {$4\vert E$};
  \node (1234) at (9,6) {$12 \lvert 34$};
  \node (e1) at (-3,11) {$E \vert 1$};
  \node (e2) at (3,11) {$E \vert 2$};
    \node (e34) at (9,11) {$E \vert 34$};
  \foreach \from/\to in {12e/e2, 3e/e1, 3e/e34, 4e/e2}
    \draw (\from) -- (\to);  
    \draw[] (1234) -- (12e);
    \draw[] (1234) -- (e34);
    \draw[] (4e) -- (e1);
    \draw[] (3e) to (e2); 
    \draw[] (4e) to (e34);
    \draw[] (12e) to (e1);

\node (o4e) at (27, 12) {$4\vert E$};
\node (oe1) at (21, 12) {$E\vert 1$};
\node (oe34) at (30 ,6) {$E \vert 34$};
\node (o3e) at (18, 6) {$3 \vert E$};
\node (o12e) at (27, 0) {$12\vert E$};
\node[right] (nontriv) at (32, 0) {$12\vert 34$};
\node (oe2) at (21, 0) {$E\vert 2$};

 \foreach \from/\to in {oe1/o4e, oe1/o3e,  oe2/o3e, oe2/o12e, oe34/o4e, oe34/o3e}
    \draw (\from) -- (\to); 

\draw (oe1) to (o12e);
\draw (oe2) to (o4e);
\fill[draw=black!100,fill=black!20,fill opacity=0.5] (30, 5.5) -- (28.1, 0) -- (32, 0) --cycle;

\end{tikzpicture}
\end{center}

\noindent Note that the biflats complex $\Delta(\BF_{M, M^\perp})$ is not pure since
    \[3 \vert E < E \vert 34  \qquad \text { and } \qquad
    12 \vert E <  12\vert 34 < E \vert 34\]
    are two maximal chains of $\BF_{M, M^\perp}$ with different lengths. These two chains form two facets of $\Delta(\BF_{M, M^\perp})$ with different dimensions.

The set of unmixed biflats is $\BF^u_{M, M^\perp} = \BF_{M, M^\perp} \setminus \{12\vert 34\}$. The corresponding poset $\BF^u_{M, M^\perp}$ and its order complex $\Delta(\BF^u_{M, M^\perp})$ are depicted below. As before, any unlabeled intersections between edges are coincidental and do not correspond to poset elements or vertices.

\begin{center}
\begin{tikzpicture} [scale=.35,auto=left]
  \node (12e) at (6,1) {$12\vert E$};
  \node (3e) at (0,1) {$3\vert E$};
  \node (4e) at (-6,1) {$4\vert E$};
  \node (e1) at (-6,11) {$E \vert 1$};
  \node (e2) at (0,11) {$E \vert 2$};
    \node (e34) at (6,11) {$E \vert 34$};
  \foreach \from/\to in {12e/e2, 3e/e1, 3e/e34, 4e/e2,  12e/e34, e1/4e, e2/3e, 12e/e1, 4e/e34}
    \draw (\from) -- (\to);    

\node (o4e) at (27, 12) {$4\vert E$};
\node (oe1) at (21, 12) {$E\vert 1$};
\node (oe34) at (30 ,6) {$E \vert 34$};
\node (o3e) at (18, 6) {$3 \vert E$};
\node (o12e) at (27, 0) {$12\vert E$};
\node (oe2) at (21, 0) {$E\vert 2$};

 \foreach \from/\to in {oe1/o4e, oe1/o3e, oe1/o12e, oe2/o4e, oe2/o3e, oe2/o12e, oe34/o4e, oe34/o3e, oe34/o12e}
    \draw (\from) -- (\to);  
\end{tikzpicture}
\end{center}

\noindent In contrast to the biflats complex $\Delta(\BF_{M, M^\perp})$, the unmixed biflats complex $\Delta(\BF^u_{M, M^\perp})$ is a pure simplicial complex.

\end{example}

We again compare these complexes with a slightly more complicated example. In this example, the poset of biflats will fail to be a ranked poset.

\begin{example}\label{ex:k4-edge}
Let the matroid $M$ defined by the graph below. The dual matroid is defined by the dual graph, which is included in the figure below.
\begin{center}
    \begin{tikzpicture}
\node at (-1, 1.25) {\Large $M$};

\node[circle,draw] (a) at (0,0) {};
\node[circle,draw] (b) at (2.5,0) {};
\node[circle,draw] (c) at (0,2.5) {};
\node[circle,draw] (d) at (2.5,2.5) {};

\draw (c) to node[above]{$1$} (d) ;
\draw (c) to node[left]{$2$} (a) ;
\draw (a) to node[below]{$3$} (b) ;
\draw (b) to node[right]{$4$} (d) ;
\draw (a) to node[left]{$5$} (d) ;

\end{tikzpicture} \hspace{1in}
        \begin{tikzpicture}
\node at (-1, 1.25) {\Large $M^\perp$};

\node[circle,draw] (a) at (0,0) {};
\node[circle,draw] (b) at (0,2.5) {};
\node[circle,draw] (c) at (3.75,1.25){};

\draw (b) to[bend left] node[above]{$1$} (c) ;
\draw (b) to[out=-45, in=175] node[above]{$2$} (c) ;
\draw (a) to[bend right] node[below]{$4$} (c) ;
\draw (a) to[out=45, in=-175] node[below]{$3$} (c) ;
\draw (b) to node[left]{$5$} (a) ;

\end{tikzpicture}
\end{center}

This matroid has two mixed biflats, $125\vert 34$ and $345\vert 12$. The Hasse diagram of the poset of biflats $\BF_{M,M^\perp}$ is drawn below. The poset of unmixed biflats is the subposet which does not include any mixed biflats. The mixed biflats are highlighted by blue rectangles.

\begin{center}
\scalebox{0.8}
{\begin{tikzpicture}
\node (1E) at (1,0) {\Large $1\vert E$};
\node (2E) at (3, 0) {\Large $2\vert E$};
\node (3E) at (5,0) {\Large $3\vert E$};
\node (4E) at (7,0) {\Large $4\vert E$};
\node (5E) at (9,0) {\Large $5\vert E$};

\node (125E) at (0, 3) {\Large $125\vert E$};
\node (13E) at (2, 3) {\Large $13\vert E$};
\node (14E) at (4, 3) {\Large $14\vert E$};
\node (23E) at (6, 3) {\Large $23\vert E$};
\node (24E) at (8, 3) {\Large $24\vert E$};
\node (345E) at (10,3) {\Large $345\vert E$};

\node[draw=blue] (12534) at (-0.1, 5) {\Large $125\vert 34$};
\node[draw=blue] (34512) at (10.1,5) {\Large $345\vert 12$};

\node (E5) at (5, 7) {\Large $E\vert 5$};
\node (E12) at (10, 7) {\Large $E\vert 12$};
\node (E34) at (0, 7) {\Large $E\vert 34$};

\draw[red,dashed] (125E) to[bend left=45] (E34);
\draw[red,dashed] (345E) to[bend right=45] (E12);
\draw (5E)--(345E);
\draw (5E)--(125E);
\draw (4E)--(345E);
\draw (4E)--(24E);
\draw (4E)--(14E);
\draw (3E)--(345E);
\draw (3E)--(23E);
\draw (3E)--(13E);
\draw (345E)--(34512);
\draw (345E)--(E5);
\draw (345E)--(E34);
\draw (34512)--(E12);
\draw (2E)--(24E);
\draw (2E)--(23E);
\draw (2E)--(125E);
\draw (24E)--(E12);
\draw (24E)--(E5);
\draw (24E)--(E34);
\draw (23E)--(E12);
\draw (23E)--(E5);
\draw (23E)--(E34);
\draw (1E)--(14E);
\draw (1E)--(13E);
\draw (1E)--(125E);
\draw (14E)--(E12);
\draw (14E)--(E5);
\draw (14E)--(E34);
\draw (13E)--(E12);
\draw (13E)--(E5);
\draw (13E)--(E34);
\draw (125E)--(E12);
\draw (125E)--(E5);
\draw (125E)--(12534);
\draw (12534)--(E34);
\end{tikzpicture}}
\end{center}

This matroid has $16$ bichains which are not biflags. Of these, the bichains \[\{125\vert E \leq E\vert 34\}  \quad \text{and} \quad \{345\vert E \leq E\vert 12\}\]
are inclusion-minimal. 
Thus the conormal complex $\Delta_{M,M^\perp}$ is the simplicial complex obtained by deleting the edges corresponding to $\{125\vert E \leq E\vert 34\}$ and $\{345\vert E \leq E\vert 12\}$ from $\Delta\left(\BF_{M,M^\perp}\right)$. 
The faces of $\Delta_{M,M^\perp}$ correspond to chains in the poset $\BF_{M, M^\perp}$ (depicted above) which do not contain a red (dashed) relation as a subchain. Notice that all maximal chains, which don't contain a red (dashed) relation as a subchain, have length $3$. This means $\Delta_{M,M^\perp}$ is a $2$ dimensional simplicial complex.

Finally, note that $\BF_{M,M^\perp}$ is not a ranked poset. This is evidenced by the saturated chains of biflats
\[
   1 \vert E < 125 \vert E < 125 \vert 34 < E \vert 34  \qquad \text { and } \qquad
   1 \vert E < 13 \vert E <  E \vert 34.
\]

\end{example}

Our last example is the most complicated. Here, the conormal complex will fail to be flag and the biflats complex will fail to be non-pure shellable or even sequentially Cohen--Macaulay. 

\begin{example}
  \label{ex:big_triangle}
Let $M$ and its dual be defined by the following graphs.

\begin{center}
\begin{tikzpicture} [scale=.8]
\node at (-1, 0.5) {\Large $M$};

\node[circle,draw] (a) at (0,0) {};
\node[circle,draw] (b) at (4,0) {};
\node[circle,draw] (c) at (2,2){};
\node[circle,draw] (d) at (5,2){};

\draw (a) to[bend right] node[below]{$1$} (b) ;
\draw (a) to[bend left] node[above]{$2$} (b) ;
\draw (a) to node[left]{$3$} (c) ;
\draw (b) to node[right]{$4$} (c) ;
\draw (b) to node[right]{$5$} (d);
\draw (c) to node[above]{$6$} (d);

\end{tikzpicture} \hspace{1in}
\begin{tikzpicture}[scale=.8]
\node at (-1, 0.5) {\Large $M^\perp$};

\node[circle,draw] (a) at (0,0) {};
\node[circle,draw] (b) at (4,0) {};
\node[circle,draw] (c) at (2,2){};
\node[circle,draw] (d) at (5,2){};

\draw (a) to[bend right] node[below]{$6$} (b) ;
\draw (a) to[bend left] node[above]{$5$} (b) ;
\draw (a) to node[left]{$4$} (c) ;
\draw (b) to node[right]{$3$} (c) ;
\draw (b) to node[right]{$1$} (d);
\draw (c) to node[above]{$2$} (d);

\end{tikzpicture}
\end{center}

Its poset of biflats $\BF_{M, M^\perp}$ is below. 

\begin{center}
\scalebox{.6}{
\begin{tikzpicture}
\node (12e) at (0,0) {\Large $12\vert E$};
\node (3e) at (3,0) {\Large $3\vert E$};
\node (4e) at (6,0) {\Large$4\vert E$}; 
\node (5e) at (9, 0) {\Large $5 \vert E$};
\node (6e) at (12, 0) {\Large $6 \vert E$};

\node (12_3456) at (-3, 3) {\Large $12 \vert 3456$};
\node (1234e) at (0, 3) {\Large $1234 \vert E$};
\node (125e) at (3, 3) {\Large $125 \vert E$};
\node (126e) at (6,3) {\Large $126 \vert E$};
\node (35e) at (9, 3) {\Large $35 \vert E$};
\node (36e) at (12, 3) {\Large $36 \vert E$};
\node (456e) at (15,3) {\Large $456 \vert E$};

\node (1234_156) at (1.75, 6) {\Large $1234 \vert 156$};
\node (1234_256) at (-1, 6) {\Large $1234 \vert 256$};
\node (1234_3456) at (-10, 6) {\Large $1234 \vert 3456$};
\node (125_3456) at (-7,6) {\Large $125 \vert 3456$};
\node (126_3456) at (-4, 6) {\Large $126 \vert 3456$};
\node (456_123) at (15, 6) {\Large $456 \vert 123$};

\node (1234_56) at (-5, 9) {\Large $1234 \vert 56$};
\node (e123) at (15, 9) {\Large $E \vert 123$};
\node (e14) at (12, 9) {\Large $E \vert 14$};
\node (e156) at (6,9) {\Large $E \vert 156$};
\node (e24) at (9, 9) {\Large $E \vert 24$};
\node (e256) at (3, 9) {\Large $E \vert 256$};
\node (e3456) at (-2,9) {\Large $E\vert 3456$};

\node (e1) at (12, 12) {\Large $E \vert 1$};
\node (e2) at (9, 12) {\Large $E \vert 2$};
\node (e3) at (6,12) {\Large $E \vert 3$};
\node (e4) at (3, 12) {\Large $E \vert 4$};
\node (e56) at (0, 12) {\Large $E \vert 56$};

 
  \foreach \from/\to in 
{12e/12_3456, 12e/1234e, 12e/125e, 12e/126e, 
3e/1234e, 3e/35e, 3e/36e, 
4e/1234e, 4e/456e, 
5e/125e, 5e/35e, 5e/456e, 
6e/126e, 6e/36e, 6e/456e,
12_3456/1234_3456, 12_3456/125_3456, 12_3456/126_3456, 
1234e/1234_156, 1234e/1234_256, 1234e/1234_3456,
125e/125_3456, 126e/126_3456, 456e/456_123,
1234_156/1234_56, 1234_256/1234_56, 
456_123/e123,
1234_3456/e3456, 126_3456/e3456, 125_3456/e3456, 
1234_156/e156, 1234_256/e256, 1234_3456/1234_56,
1234_56/e56, 
e156/e1, e156/e56, 
e256/e2, e256/e56, 
e3456/e3,  e3456/e4, e3456/e56, 
e14/e1, e14/e4, e24/e2, e24/e4,
e123/e1, e123/e2, e123/e3}
    \draw (\from) -- (\to);  

  \foreach \from/\to in 
{1234e/e24, 1234e/e14, 1234e/e123,
125e/e256, 125e/e156, 125e/e24, 125e/e14, 125e/e123,
126e/e256, 126e/e156, 126e/e24, 126e/e14, 126e/e123,
35e/e3456, 35e/e256, 35e/e156, 35e/e24, 35e/e14, 35e/e123,
36e/e3456, 36e/e256, 36e/e156, 36e/e24, 36e/e14, 36e/e123,
456e/e3456, 456e/e256, 456e/e156, 456e/e24, 456e/e14}
    \draw (\from) -- (\to);  

  \foreach \from/\to in 
{126e/126_3456,126_3456/e3456,e3456/e56,126e/e256,e256/e56}
\draw[] (\from) -- (\to);
\end{tikzpicture}
}
\end{center}

The conormal complex of this matroid, $\Delta_{M, M^\perp}$ is \emph{not flag}. To see this, note that $\{12\vert E \leq 1234\vert 3456\}$, $\{1234\vert 3456\leq E\vert 56\}$, and $ \{12\vert E \leq E\vert 56\}$ 
are all biflags but
\[\{12\vert E \leq 1234\vert 3456 \leq E\vert 56\}\]
is not a biflag. This amounts to having the boundary of a triangle in the conormal complex $\Delta_{M,M^\perp}$ which is not filled in.

It can be checked computationally that the biflats complex $\Delta(\BF_{M,M^\perp})$ is not \emph{non-pure shellable}. Non-pure shellability is an analogue of shellability for non-pure complexes introduced by Bj\"orner and Wachs \cite{BW}. Like the usual notion of shellability, non-pure shellability has various algebraic and numerical implications. Specifically,
\[\text{non-pure shellability} \implies \text{sequentially Cohen--Macaulay} \implies \text{positive $h$-triangle.} \]

The $h$-triangle of a simplicial complex $\Delta$ is a triangular vector of numbers constructed from $\Delta$. If any of these numbers are negative, $\Delta$ cannot be non-pure shellable nor sequentially Cohen--Macaualy.  
Below, we compute that the $h$-triangle of $\Delta(\BF_{M,M^\perp})$  is
\[\begin{matrix*}
  0&   &   &    &   &\\
  0& 0 &   &    &   & \\
  0& 0 & 0 &    &   &\\
  0& 0 & 0 &  0 &   & \\
  0& 4 & 58&  76& 16& \\
  1& 21& 46& -11&  4& 0.
\end{matrix*}\]
As this $h$-triangle contains a negative entry, the biflats complex is not non-pure shellable nor sequentially Cohen--Macaulay.
\end{example}

\begin{remark}
  \label{rmk:shellable}
In each of the above examples, the conormal complex is a shellable simplicial complex. As the conormal complex of a matroid $M$ is homeomorphic to $\Delta_M*\Delta_{M^\perp}$, it has the homotopy type of a shellable simplicial complex. It is unknown whether all conormal complexes are shellable.
\end{remark}

\subsection{Fan structures}\label{sec:conormal_back}
The conormal complex $\Delta_{M,M^\perp}$ of a matroid $M$ is the simplicial complex associated to the \emph{conormal fan} of $M$, denoted $\Sigma_{M,M^\perp}$, which was introduced and studied by Ardila, Denham and Huh in \cite{ADH} and \cite{ADHCombo}. In this section, we give a short introduction to the conormal fan of a matroid as well as a geometric proof of the fact that $\Delta_{M,M^\perp}$ is homeomorphic to the join of Bergman complexes $\Delta_M * \Delta_{M^\perp}$ (Proposition \ref{prop:homeomorphic}). No tropical geometry background is needed to understand this section. Although our terminology is motivated by tropical geometry, we will not use in any technical results from the field.

As the conormal fan is a particular simplicial fan, we start by recalling the definition of a simplicial fan. Let $V$ be a finite dimensional real vector space. A $k$-dimensional simplicial cone  $\sigma \subset V$ is a subset that can be written as 
\[\sigma= \cone\{v_1,\ldots, v_k\} \coloneqq \{\lambda_1v_1+\ldots + \lambda_k v_k : \lambda_i\geq 0 \}\subset V\] 
where $v_1,\ldots, v_k$ are $k$ linearly independent vectors in $V$. A cone $\tau$ is a \emph{face} of $\sigma$ if it can be written as $\tau = \cone\{v_{i_1},\ldots, v_{i_\ell}\}$ for some subset $\{v_{i_1},\ldots, v_{i_\ell}\} \subset \{v_1,\ldots, v_k\}$. We declare $\cone(\emptyset)= \{0\}$ so that $\{0\}$ is a face of every cone. A \emph{simplicial fan} $\Sigma$ in $V$ is a non-empty collection of simplicial cones $\sigma \in \Sigma$ in $V$ such that
\begin{itemize}
\item If $\sigma \in \Sigma$ and $\tau$ is a face of $\sigma$, then $\tau \in \Sigma$, and
\item Every two cones $\sigma, \tau \in \Sigma$ intersect exactly along a common face.
\end{itemize}
We say that $\Sigma$ is \emph{complete} if the union of all of its cones is equal to the whole vector space $V$.

The conormal fan is a restriction of a larger fan called the \emph{bipermutohedral fan} $\Sigma_{E,E}$. Therefore, we first define the bipermutohedral fan $\Sigma_{E,E}$. The bipermutohedral fan is a complete, simplicial fan inside the product of \emph{tropical projective tori} whose cones correspond to biflags of \emph{bisubsets}. 
\begin{defn}
Given a ground set $E$, the \emph{tropical projective torus} $N_E$ is the $(n-1)$-dimensional real vector space
\[\R^E/\R\mathbf{e}_E, \quad\quad \mathbf{e}_E\coloneqq\sum_{i\in E}e_i. \]
\end{defn}
If we relax our definition of biflat to not just consist of flats but instead any non-empty subset of the ground set, we arrive at the definition of a bisubset.
\begin{defn}
A \emph{bisubset} $S\vert T$ consists of non-empty subsets $S,T\subseteq E$ such that at least one of them is not equal to $E$ and
\[ S\cup T = E.\]
\end{defn}
\noindent Just like biflats, we can form bichains and biflags of bisubsets. The definitions for bichains and biflags of bisubsets are identical to those given in Definitions ~\ref{biflat_order} and ~\ref{def:biflag}.

Following Section $2.6$ of \cite{ADH}, we define the bipermutohedral fan by its rays and cones. For any subset $S\subseteq E$, define the vectors $\mathbf{e}_S$ and $\mathbf{f}_S$ to be the image of $\sum_{i\in S}e_i$ in the first and second copy of $N_E$, respectively. Then define $\mathbf{e}_{S|T}$ to be the vector $(\mathbf{e}_S, \mathbf{f}_T)$ in $N_E\times N_E$.
\begin{defn}
The \emph{bipermutohedral fan} $\Sigma_{E,E}$ is the $(2n-2)$-dimensional simplicial fan in $N_E\times N_E$ with cones
\[\sigma_{\mathcal{S}\vert \mathcal{T}} = \cone\left\{\mathbf{e}_{S\vert T}\right\}_{S\vert T\in \mathcal{S}\vert \mathcal{T}}, \quad\text{for biflags of bisubsets }\mathcal{S}\vert \mathcal{T}. \]
\end{defn}

\noindent The bipermutohedral fan is a rich combinatorial object with many different interpretations. For an exploration of these interpretations and their consequences, see Section $2$ of \cite{ADH}.

\begin{defn}
The tropical linear space of a matroid is defined to be
\[\trop(M) = \left\{ z \in N_E :  \min_{i\in C}(z_i) \text{ is achieved at least twice for every circuit $C$ of $M$} \right\}\subseteq N_E. \]
\end{defn}

\begin{defn}
The \emph{conormal fan} of $M$, denoted $\Sigma_{M,M^\perp}$, is the intersection of $\trop(M)\times \trop(M^\perp)$ with $\Sigma_{E,E^\perp}$. By Proposition $3.6$ of \cite{ADH}, $\Sigma_{M,M^\perp}$ is the subfan of $\Sigma_{E,E^\perp}$ whose cones are
\[\sigma_{\FG} =  \cone\left\{\mathbf{e}_{F\vert G}\right\}_{F\vert G\in \mathcal{F}\vert \mathcal{G}}, \text{  for biflags of biflats $\F\vert\G$ of $M$.}\] 
\end{defn}

\begin{remark} \label{rem:loops}
If $M$ contains a loop $a\in E$ then $\trop(M)$ will be empty since $\{a\}$ will be a circuit and thus \[\trop(M) \subseteq \left\{ z \in N_E :  \min_{i\in \{a\}}(z_i) \text{ is achieved at least twice.} \right\} = \emptyset. \] 
Similarly if $M$ contains a coloop, $M^\perp$ will contain a loop and $\trop\left(M^\perp\right)$ will be empty. Thus if $M$ contains a loop or a coloop then $\trop(M)\times \trop(M^\perp)$ will be empty. Hence, we require our matroids to be loopless and coloopless to eliminate the case when the the conormal fan of $M$ has empty support.
\end{remark}

We consider one more subdivision of $\trop(M)\times \trop\left(M^\perp\right)$. This subdivision is the product of the \emph{Bergman fans} of $M$ and $M^\perp$. The Bergman fan of a matroid, first introduced in \cite{AK}, is a well studied fan whose support is $\trop(M)$. 
\begin{defn}
The \emph{Bergman fan} $\Sigma_M$ of a matroid $M$ is the simplicial fan in $N_E$ whose support is $\trop(M)$ and cones are
\[\sigma_{\mathcal{F}} =  \cone\left\{\mathbf{e}_{F}\right\}_{F\in \mathcal{F}}, \text{  for flags of flats $\mathcal{F}$ of $M$.} \]
\end{defn}
\noindent As $\Sigma_M$ is a subdivision of $\trop(M)$ and $\Sigma_{M^\perp}$ is a subdivision of $\trop\left(M^\perp\right)$, $\Sigma_M\times \Sigma_{M^\perp}$ will be a subdivision of $\trop(M)\times \trop\left(M^\perp\right)$.

To every simplical fan, there is an associated simplical complex whose faces are the cones of the fan. This complex is combinatorially equivalent to the complex obtained by intersecting the fan with a sphere centered around the origin. The conormal complex $\Delta_{M,M^\perp}$ is the simplical complex associated with the conormal fan $\Sigma_{M,M^\perp}$ while $\Delta_M*\Delta_{M^\perp}$ is the simplicial complex associated with the product of Bergman fans $\Sigma_M\times \Sigma_{M^\perp}$.

This section is summarized in the following diagram. Here, the downward diagonal arrows represent subdividing $\trop(M)\times \trop\left(M^\perp\right)$ with a fan structure. The vertical maps are given by intersecting the fan with a sphere around the origin in $N_E\times N_E$. The map $\varphi$ is the inclusion given by viewing biflags as bichains while the map $\psi$ is the simplicial isomorphism of $\Delta_M*\Delta_M$ with the unmixed biflats complex. 
\[\begin{tikzcd}
	& {\trop(M)\times\trop\left(M^\perp\right)} \\
	{\Sigma_{M,M^\perp}} && {\Sigma_M\times \Sigma_{M^\perp}} \\
	{\Delta_{M,M^\perp}} & {\Delta(\BF_{M,M^\perp})} & {\Delta_M*\Delta_{M^\perp}}
	\arrow[squiggly, from=2-1, to=3-1]
	\arrow["\varphi", hook, from=3-1, to=3-2]
	\arrow[squiggly, from=2-3, to=3-3]
	\arrow["\psi"', hook', from=3-3, to=3-2]
	\arrow[squiggly, from=1-2, to=2-1]
	\arrow[squiggly, from=1-2, to=2-3]
\end{tikzcd}\]

In the Sections ~\ref{sec:unmixed} and ~\ref{sec:conormalcomplex}, we show that $\varphi$ and $\psi$ induce different deformation retracts of $\Delta\left(\BF_{M,M^\perp}\right)$ onto $\Delta_{M,M^\perp}$ and $\Delta_M * \Delta_{M^\perp}$ by giving different sequences of elementary collapses.

From the above discussion, the following proposition follows immediately.
\begin{prop}
  \label{prop:homeomorphic}
The conormal complex of a matroid $\Delta_{M,M^\perp}$ is homeomorphic to the join of Bergman complexes $\Delta_M * \Delta_{M^\perp}$.
\end{prop}

\section{Unmixed biflats complex}\label{sec:unmixed}

For this section, we fix a loopless and coloopless matroid $M$ on ground set $E$ and let $M^\perp$ be its dual matroid. We drop subscripts and let $\BF$ refer to the poset of biflats of $M$. 
In this section, we prove the first of our main theorems.
\begin{thmSpecial}
    Given a matroid $M$, there is an explicit sequence of elementary collapses
    \[\Delta(\BF)\diagarrow \Delta_1 \diagarrow \cdots \diagarrow \Delta(\BF_{M,M^\perp}^u)\cong \Delta_M * \Delta_{M^\perp} \]
    from the biflats complex $\Delta(\BF_{M,M^\perp})$ onto the join of the Bergman complexes $\Delta_M * \Delta_{M^\perp}$. 
\end{thmSpecial}

\noindent Recall that Theorem \ref{thm:unmixed} implies that $\Delta(\BF)$ is simple homotopy equivalent to $\Delta(\BF_{M,M^\perp})$. Our sequence of elementary collapses will come from ordering the mixed biflats and deleting them, one at a time, from the biflats complex using Lemma \ref{lem:pos_collapse}.

Let $\BF^m\coloneqq \BF\setminus \BF^{u}$ be the set of mixed biflats. We give $\BF^m$ a poset structure via the inclusion $\BF^m\subseteq \BF$. Let $\lessdot$ be any linear extension of the partial ordering $\leq$ on $\BF^m$. That is, $\lessdot$ is a total ordering of $\BF^m$ such that if $F\vert G$ and $F'\vert G'$ are mixed biflats with $F\subseteq F'$ and $G\supseteq G'$ then $F\vert G \lessdot F'\vert G'$.  

The total ordering $\lessdot$ will dictate the order in which we delete the mixed biflats from $\Delta(\BF)$. We will first illustrate this with an example and then prove the general case.

\begin{example} \label{ex:mixed_removal}
Let $M$ be the matroid defined in Example \ref{ex:big_triangle}. We want to delete all of the mixed biflats from $\Delta(\BF)$ through a sequence of elementary collapses, in accordance with the total order $\lessdot$. The subposet of mixed biflats, $\BF^m$ is shown below.
\begin{center}
    \scalebox{0.7}{\begin{tikzpicture}
\node (12_3456) at (-3, 0) {\Large $12 \vert 3456$};

\node (1234_156) at (3, 2) {\Large $1234 \vert 156$};
\node (1234_256) at (6, 2) {\Large $1234 \vert 256$};
\node (1234_3456) at (0, 2) {\Large $1234 \vert 3456$};
\node (125_3456) at (-6,2) {\Large $125 \vert 3456$};
\node (126_3456) at (-3, 2) {\Large $126 \vert 3456$};
\node (456_123) at (9, 2) {\Large $456 \vert 123$};
\node (1234_56) at (3, 4) {\Large $1234 \vert 56$};

 
  \foreach \from/\to in 
{12_3456/1234_3456, 12_3456/125_3456, 12_3456/126_3456, 
1234_156/1234_56, 1234_256/1234_56, 1234_3456/1234_56}
    \draw (\from) -- (\to);  
\end{tikzpicture}}
\end{center}

Suppose that $\lessdot$ is a total order such that $12\vert 3456 \lessdot 125\vert 3456 \lessdot \ldots \lessdot 1234\vert 56$.

We will first delete $12\vert 3456$ from $\BF$. From the Hasse diagram in Example \ref{ex:big_triangle}, we can check that $12\vert E$ is the unique maximal biflat smaller than $12\vert 3456$. Thus, by Lemma \ref{lem:pos_collapse}, there is a sequence of elementary collapses from $\Delta(\BF)$ onto $\Delta(\BF\setminus \{12\vert 3456\})$. 

Next we will delete $125\vert 3456$ from $\BF\setminus \{12\vert 3456\}$. First, note that we cannot use Lemma \ref{lem:pos_collapse} to delete $125\vert 3456$ from $\BF$. This is because the poset $\BF_{<125\vert 3456}$ has two maximal elements, $12\vert 3456$ and $125\vert E$. But, as we just deleted $12\vert 3456$ from $\BF$, the poset $(\BF\setminus \{12\vert 3456\})_{<125\vert 3456}$ will have $125 \vert E$ as its unique maximal element. Again, this lets us use Lemma \ref{lem:pos_collapse} to collapse $\Delta(\BF\setminus \{12\vert 3456\})$ onto $\Delta(\BF\setminus \{12\vert 3456, 125\vert 3456\})$.

This pattern will continue for the rest of the mixed biflats. When it comes time to delete a mixed biflat $F\vert G$, there will be no remaining mixed biflats that are smaller than $F\vert G$ and $F\vert E$ will be the largest remaining biflat that is smaller than $F\vert G$. This will let us use Lemma \ref{lem:pos_collapse} to collapse onto the deletion of $F\vert G$.
\end{example}

We formalize the ideas from the above example using the following proposition. This proposition is the key inductive step of the proof of Theorem \ref{thm:unmixed}.

\begin{prop}
\label{prop:unmixed_inductive}
Let $F\vert G$ be a mixed biflat and $\mathbf{P}= \BF \setminus \{ F'\vert G' \in \BF^m : F'\vert G' \lessdot F\vert G\}$ be the sub-poset of $\BF$ obtained by removing all mixed biflats smaller than $F\vert G$ under the ordering $\lessdot$. Then there is a sequence of elementary collapses from $\Delta(\mathbf{P})$ onto $\Delta(\mathbf{P}\setminus \{F\vert G\})$.
\end{prop}
\begin{proof}
By Lemma \ref{lem:pos_collapse}, it suffices to show that $\mathbf{P}_{< F\vert G} = \{H\vert K \in \mathbf{P}: H\vert K \lneq F\vert G\}$ has a unique maximal element. Let $H\vert K \in \mathbf{P}_{<F\vert G}$, then, by construction of $\lessdot$, we know that $H\vert K$ is not a mixed biflat. But we also know that $H\subseteq F \subsetneq E$, so it must be that $K=E$. This implies that
 \[\mathbf{P}_{<F\vert G} = \{ H\vert E \in \BF: H\subseteq F\}. \]
This poset clearly has a unique maximal element, namely $F\vert E$.
\end{proof}

We proceed to prove Theorem ~\ref{thm:unmixed} below.

\begin{proof}
Let $\BF^m = F_1\vert G_1 \lessdot F_2\vert G_2 \lessdot \ldots \lessdot F_\ell\vert G_\ell$. By Proposition \ref{prop:unmixed_inductive}, there is a sequence of elementary collapses from $\Delta(\BF)$ onto $\Delta(\BF\setminus \{F_1\vert G_1\})$ and from $\Delta(\BF\setminus \{F_1\vert G_1\})$ onto $\Delta(\BF\setminus \{F_1\vert G_1, F_2\vert G_2\})$ and so on. This gives a sequence of elementary collapses
\[\Delta(\BF) \diagarrow\cdots \diagarrow \Delta(\BF\setminus \{F_1\vert G_1\}) \diagarrow \cdots \diagarrow \Delta(\BF\setminus \BF^m)= \Delta(\BF^u). \]

By Lemma \ref{lem:unmixed_join}, we know that $\Delta(\BF^u)$ is isomorphic to $\Delta_M*\Delta_{M^\perp}$ and therefore we have proven our claim.
\end{proof}

\section{Conormal Complex}\label{sec:conormalcomplex}
As in the last section, we fix $M$, a loopless and coloopless matroid on ground set $E$ and let $M^\perp$ be its dual matroid. We  drop subscripts and let $\BF$ refer to the poset of biflats of $M$. In this section, we set up and prove our second theorem.

\begin{thmSpecial}
\label{thm:thm2}
    Given a matroid $M$, there is an explicit sequence of elementary collapses
    \[\Delta(\BF)\diagarrow \Delta_1 \diagarrow \cdots \diagarrow  \Delta_{M,M^\perp}\]
    from the biflats complex $\Delta(\BF_{M,M^\perp})$ onto the conormal complex $\Delta_{M,M^\perp}$. 
\end{thmSpecial}

Define $\mathbf{BC}$ to be the collection of bichains which aren't biflags. We will prove that $\Delta(\BF)$ and $\Delta_{M,M^\perp}$ are simple homotopy equivalent by performing elementary collapses to remove all of $\mathbf{BC}$ from $\Delta(\BF)$. 
We do so by using Lemma \ref{lem:collapse} to collapse the biflats poset to a complex that results from deleting a subset of bichains that are not biflags from $\mathbf{BF}$.





\begin{defn} \label{def:gamma}
Let $\mathbf{C}\subseteq \mathbf{BC}$ be the subset of bichains $\F\vert \G$ which are not biflags and no flat or dual flat appears twice in $\F\vert \G$. In other words, $\mathbf{C}$ contains all bichains of $\mathbf{BC}$ with the form
\[F_1\vert E \leq F_2\vert G_2 \leq \ldots \leq F_{\ell-1}\vert G_{\ell-1} \leq E\vert G_\ell \]
where 
\[F_1 \subsetneq F_2 \subsetneq \ldots \subsetneq F_{\ell-1} \subsetneq E \quad \text{ and } \quad E \supsetneq G_2 \supsetneq \ldots \supsetneq G_{\ell-1} \supsetneq G_{\ell}\]
are strict chains of flats and dual flats.
We refer to $\mathbf{C}$ as the set of \emph{non-repeating bichains}.
\end{defn}

\begin{example} Let $M$ be the matroid defined in Example ~\ref{ex:small_triangle}.
There are two bichains of $M$ which are not biflags:
\[12\vert E\leq E\vert 34 \quad \text{ and } \quad 12\vert E\leq 12\vert 34 \leq E\vert 34. \]
Of these, only $12\vert E \leq E\vert 34$ does not contain a repeated flat or dual flat, so in this case $\mathbf{C}= \{12\vert E \leq E\vert 34\}$.
\end{example}

The following lemma will allow us remove the set of bichains $\mathbf{BC}$ by removing only the elements in $\mathbf{C}$.

\begin{lem}
\label{lem:gamma}
If $\F\vert \G$ is a bichain of $\mathbf{BC}$ and a flat or dual flat appears twice in $\F\vert \G$, then there exists some $\F'\vert \G'\in \mathbf{BC}$ which is strictly contained in $\F\vert \G$. 
\end{lem}

\begin{proof}
Suppose without loss of generality that $\F\vert \G$ contains a repeated flat $F$ and thus the bichain $\F\vert \G$ has biflats $F\vert G, F\vert G'$ with $G\supsetneq G'$. Recall that the bichain $\F\vert \G$ is not a biflag exactly when
\[\bigcup_{H\vert K \in \F\vert \G} H\cap K = E. \] 
Then $ F\cap G'\subseteq  F\cap G$. It follows that 
$\F\vert \G\setminus \{F\vert G'\}$ is not a biflag 
since
\[\bigcup_{H\vert K \in \F\vert \G\setminus\{F\vert G'\}} H\cap K = E. \qedhere \]
\end{proof}

Since we have shown that that every inclusion-minimal bichain of $\mathbf{BC}$ is non-repeating, we have that 
\[\Delta_{M,M^\perp}=\Delta(\BF)\setminus \mathbf{BC}=\Delta(\BF) \setminus \bigcup_{\F\vert\G \in \mathbf{C}} \str_{\Delta(\BF)}(\F\vert\G).\]

\begin{remark}
One could ask whether $\mathbf{C}$ is actually equal to the set of inclusion minimal bichains of $\mathbf{BC}$. This is not the case for larger matroids. For example, the matroid of the six vertex wheel graph has non-repeating bichains in $\mathbf{BC}$ which are not inclusion minimal. Our choice of $\mathbf{C}$ as a larger than minimal indexing set is necessary for the upcoming proofs. We will use the fact that $\mathbf{C}$ contains every non-repeating bichain in the proofs of Lemmas \ref{lem:NR_struct} and \ref{lem:NR_cone}.
\end{remark}

We now seek to define an order on the bichains $\F\vert \G \in \mathbf{C}$ by which we can sequentially collapse $\Delta(\BF)$ onto $\dell_{\Delta(\BF)}\left(\F\vert \G\right)$, and $\dell_{\Delta(\BF)}\left(\F\vert \G\right)$ onto 
$\dell_{\dell_{\Delta(\BF)}\left(\F\vert \G\right)}\left(\F'\vert \G'\right)$ and so on. 

\begin{defn}
Let
\begin{gather*}
\F\vert \G = \{F_1\vert E \leq F_2\vert G_2 \leq \dots \leq F_{\ell-1}\vert G_{\ell-1}\leq E\vert G_\ell\} \in \mathbf{C} \quad \text{and}\\
\H\vert\K = \{H_1\vert E \leq H_2\vert K_2 \leq \dots \leq H_{m-1}\vert K_{m-1}\leq E\vert K_{m}\} \in \mathbf{C}.
\end{gather*}

We define $\preceq$ to be the binary relation on non-repeating bichains
such that
$\F\vert \G\preceq \H\vert\K$ 
if and only if
\[ F_1\vert E \gneq H_1\vert E 
\quad \text{ or } \quad 
 \left(F_1\vert E = H_1\vert E 
\;  \text{ and } \;
 F_2\vert G_2 \lneq H_2\vert K_2\right) 
 \quad \text{ or } \quad 
 \F\vert \G = \H\vert\K.
\]
\end{defn}

\noindent In words, the relation $\preceq$ compares two different bichains $\F\vert\G$ and $\H\vert \G$ using the following procedure: First check if $F_1$ strictly contains $H_1$. If so, then declare $\F\vert \G \preceq \H\vert \K$. If $F_1$ equals $H_1$, then check if $F_2\vert G_2$ is strictly smaller than $H_2\vert K_2$ as biflats. If this is true, we again declare that $\F\vert \G \preceq \H\vert \K$. In any other case, we say that $\F\vert \G \not \preceq \H\vert \K$. See Example \ref{ex:preceq} for an example of this ordering.

\begin{lem}
The relation $\preceq$ is a partial order on the set of non-repeating bichains $\mathbf{C}$.
\end{lem}
\begin{proof} 
By definition, the relation $\preceq$ is reflexive. 

Now we prove anti-symmetry. Suppose that $\F\vert\G \preceq \H\vert\K$. If $F_1\vert E\gneq H_1\vert E$ then we cannot have that $\H\vert\K\preceq \F\vert \G$ since $H_1\vert E$ is strictly less than $F_1\vert E$ so none of our three conditions can hold. Similarly, if $F_1\vert E= H_1\vert E$ and $F_2\vert G_2 \lneq H_1\vert K_2$, we cannot have that $\H\vert\K \preceq \F\vert\G$. Thus if $\F\vert\G \preceq \H\vert\K$ and $\H\vert\K\preceq \F\vert\G$ we must have that $\F\vert\G = \H\vert\K$, confirming that $\preceq$ is anti-symmetric.

To check transitivity, let $\F\vert \G, \H\vert\K, \mathcal{L}\vert\mathcal{M}$ be non-repeating bichains. If $\F\vert \G \preceq \H\vert\K \preceq \mathcal{L}\vert\mathcal{M}$, the conditions of $\preceq$ ensure that 
\[F_1\vert E \geq H_1\vert E \geq L_1\vert E.\] 
If any of these inequalities are strict, transitivity of the ordering for biflats will imply $F_1\vert E \gneq L_1\vert E$ and thus $\F\vert \G\preceq \mathcal{L}\vert \mathcal{M}$. If instead $F_1\vert E = H_1\vert E = L_1 \vert E$, then the second and third conditions of $\preceq$ ensure that
\[F_2\vert G_2 \leq H_2\vert K_2 \leq L_2\vert M_2.\] 

Again, if any of these inequalities are strict, transitivity of the ordering will imply $F_2\vert G_2 \lneq L_2\vert M_2$ and thus $\F\vert \G\preceq \mathcal{L}\vert \mathcal{M}$. Finally, if all the inequalities are equalities, then the third condition of $\preceq$ will ensure $\F\vert \G = \H\vert \K=\mathcal{L}\vert \mathcal{M}$ so $\F\vert \G \preceq \mathcal{L}\vert \mathcal{M}$.
\end{proof}

\begin{example}
\label{ex:preceq}
There are $228$ bichains of the matroid $M$, as defined in Example ~\ref{ex:big_triangle}, which are not biflags. However, there are only nine non-repeating bichains. The poset structure of the non-repeating bichains $(\mathbf{C},\preceq)$ is given below:

\centering{
\scalebox{0.7}{\begin{tikzpicture}
  \node (12EE3456) at (0,8) {\Large $\{12\vert E \leq E\vert 3456\}$};
  \node (12Elong) at (0,6) {\Large $\{12\vert E \leq 1234\vert 3456 \leq E\vert 56\}$};
    \node (123456) at (-6,4) {\Large$\{1234\vert E\leq E\vert 56 \}$};
    \node (1253456) at (0,4) {\Large$\{125\vert E\leq E\vert 3456 \}$};
    \node (1263456) at (6,4) {\Large$\{126\vert E\leq E\vert 3456 \}$};
  \node (1234156) at (-6,2) {\Large$\{1234\vert E\leq E\vert 156 \}$};
  \node (1234256) at (-12,2) {\Large$\{1234\vert E\leq E\vert 256 \}$};
  \node (12343456) at (0,2) {\Large$\{1234\vert E\leq E\vert 3456\}$};
  \node (456123) at (7,7) {\Large$\{456\vert E\leq E\vert 123\}$};

    \draw[] (12EE3456) -- (12Elong);
    \draw[] (123456) -- (12Elong);
    \draw[] (1253456) -- (12Elong);
    \draw[] (1263456) -- (12Elong);
    \draw[] (1234156) -- (123456);
    \draw[] (12343456) -- (123456);
    \draw[] (1234256) -- (123456);
\end{tikzpicture}}
}

\end{example}

\begin{defn} Fix $\prec^*$ a linear extension of $\preceq$. 
Let $\Delta_{\prec\F\vert\G}$ be the subcomplex of $\Delta\left(\BF\right)$ formed by deleting the open stars of all bichains $\F'\vert \G'\in \mathbf{C}$ which are strictly smaller than $\F\vert\G$ under $\prec^*$. That is,
\[
    \Delta_{\prec \F\vert\G} \coloneqq \; \Delta(\BF) \setminus \bigcup_{\F\vert \G' \prec^* \F\vert \G}\str_{\Delta(\BF)}\left(\F'\vert \G'\right)
\]
\end{defn}

We have constructed $\preceq$ so that the smallest bichains of $\mathbf{C}$ are those with the largest first biflat and the smallest possible second biflat. This gives us strong control of the bichains $\H\vert \K$ contained in the link of $\F\vert \G$ in $\Delta_{\prec\F\vert\G}$. Crucially, we will know what the first few mixed biflats of $\H\vert \K$ look like. This is stated precisely and proven in Lemma \ref{lem:NR_struct}. We then use this lemma in order to prove that the link of $\F\vert \G$ in $\Delta_{\prec\F\vert\G}$ has a cone vertex. This cone vertex will ``fit in-between'' the first and second biflat of $\F\vert \G$. After proving this, we will be able to collapse $\Delta_{\prec \F\vert \G}$ onto $\dell_{\Delta_{\prec \F\vert \G}}(\F\vert \G)$. This will be a vital step in our inductive proof of Theorem \ref{thm:thm2}. 
\begin{lem}
\label{lem:NR_struct}
Let $\F\vert \G$ be a non-repeating bichain written as
\[\F\vert \G = \{F_1\vert E \leq F_2\vert G_2 \leq \dots \leq F_{\ell-1}\vert G_{\ell-1}\leq E\vert G_\ell\}.\]
If $\H\vert \K$ is a bichain in $\Delta_{\prec\F\vert \G}$ containing $\F\vert \G$ then $\H\vert \K$ has the following structural properties:
\begin{enumerate}[(a)]
\item \label{lem:NR_struct_1} $F_1\vert E$ is the largest unmixed biflat of $\H\vert \K$ with $E$ as its dual flat. 
\item \label{lem:NR_struct_2} If $F_1\vert E \lneq H\vert K \lneq F_2\vert G_2$ with $H\vert K\in \H\vert \K$ then $H=F_1$ or $K=G_2$. 
\end{enumerate}
\end{lem}

\noindent{\it Proof of (\ref{lem:NR_struct_1}).}
Suppose that $\H\vert \K\in \Delta_{\prec\F\vert \G}$ contains some biflat $F'\vert E$ where $F'\vert E\gneq F_1\vert E$. This means that $F'\supsetneq F_1$.

First, compare $F'\vert E$ and  $F_2 \vert G_2$.  Since $\F\vert \G\in\mathbf{C}$, we know that $G_2\neq E$; in other words $G_2\subsetneq E$. It follows that $F'\vert E\leq F_2\vert G_2$. In particular, we know that $F'\subseteq F_2$. We break into two cases, the first where $F'\subsetneq F_2$ and the second where $F'= F_2$.

If $F'\subsetneq F_2$, then the length $\ell$ bichain
\[\F'\vert \G' \coloneqq \{F'\vert E \leq F_2\vert G_2 \leq F_3\vert G_3 \leq \dots \leq F_{\ell-1}\vert G_{\ell-1} \leq E\vert G_\ell\}\]
is not a biflag since $\F\vert \G$ is not a biflag and $F'\cup G_2 \supseteq F_1\cup G_2 = E$. 
Further the biflats of $\F'\vert \G'$ have no repeated flats or dual flats since $\F\vert \G$ was an element of $\mathbf{C}$ and we assumed that $F'\subsetneq F_2$. 
Comparing the two bichains as elements of the poset of non-repeating bichains ($\mathbf{C}, \preceq)$, we see that  $F'\vert E \gneq F_1\vert E$ implies that $\F'\vert \G' \prec^* \F\vert \G$. 
We conclude that $\F'\vert \G'$ is not a chain in the simplicial complex $\Delta_{\prec\F\vert \G}$ because it is in the set of bichains that have been deleted from $\BF$ in the construction of this complex.
This contradicts our assumption of $\H\vert \K$ being a bichain in $ \Delta_{\prec\F\vert \G}$  since $\F'\vert \G'\subseteq \H\vert \K$.

On the other hand, if $F'=F_2$, then the length $\ell-1$ bichain
\[\F'\vert \G' \coloneqq  \{F'\vert E \leq F_3\vert G_3 \leq \ldots \leq F_{\ell-1}\vert G_{\ell-1} \leq E\vert G_\ell\}\]
is not a biflag. We can check this using the gap condition for biflags. Because $\F\vert \G$ is not a biflag, the only possible gap of $\F'\vert \G'$ is at the first index. However, 
\[F'\cup G_3 = F_2\cup G_3 = E\] 
and thus $\F'\vert \G'$ contains no gaps and is not a biflag. Furthermore, $\F'\vert \G'$ has no repeated flats or dual flats 
by virtue of $\F\vert \G$ being a non-repeating bichain.
As elements of $(\mathbf{C}, \preceq)$, we have that $\F'\vert \G' \prec^* \F\vert \G$ 
since $F'\vert E \gneq F_1\vert E$.
As in the previous case, the bichain $\F'\vert \G'$ cannot be an element of $\Delta_{\prec\F\vert \G}$.
This contradicts our starting hypothesis that $\H\vert \K\in \Delta_{\prec\F\vert \G}$.

Hence, the largest biflat unmixed in $\H\vert\K$ must be $F_1 \vert E$. 
\qed

\noindent{\it Proof of (\ref{lem:NR_struct_2}).}
Assume that $\H\vert \K$ contains some biflat $H\vert K$ such that
\[F_1\vert E \lneq H\vert K \lneq F_2\vert G_2, \quad H\neq F_1 \quad \text{and}\quad K\neq G_2.\] 
Since $F_1\vert E \lneq H\vert K\lneq F_2\vert G_2$, we must have that $F_1 \subsetneq H$ and $K\supsetneq G_2$. By  the first part of this lemma,
we also have that $K\subsetneq E$. We break into two cases depending on whether $H\cup G_3 =E$ or $H\cup G_3 \neq E$.

First, suppose that $H\cup G_3 = E$. Then $\H\vert \K$ contains the length $\ell$ bichain
\[\F'\vert \G' \coloneqq \{F_1\vert E \leq H\vert K \leq F_3\vert G_3 \leq \ldots \leq F_{\ell-1}\vert G_{\ell-1} \leq E\vert G_\ell\}.\]

\noindent To affirm $\F'\vert \G'$ is not a biflag, we check the biflats in $\F'\vert \G'$ that are distinct from $\F\vert\G$ to see that no gaps were introduced. Notice that $F_1\cup K \supseteq F_1\cup G_2 = E$ and  
$H\cup G_3 = E$, thus $\F'\vert \G'$ not a biflag. 
Further, because $\F\vert\G$ has no repeated flats or coflats, we have that
\[F_1\subsetneq H \subseteq F_2 \subsetneq F_3 
\qquad \text{and} \qquad  E\supsetneq K \supseteq G_2 \supsetneq G_3.\]
This implies that the bichain $\F'\vert \G'$ has no repeated flats or coflats.
As elements of $(\mathbf{C}, \preceq)$, we have that $\F'\vert \G'\prec^* \F\vert \G$ since $H\vert K \lneq F_2\vert G_2$.
It follows that $\F'\vert \G' \notin \Delta_{\prec\F\vert \G}$.  
This contradicts our assumption that  $\H\vert \K\in \Delta_{\prec\F\vert \G}$, since $\F'\vert \G'$ is a sub-bichain of $\H\vert \K$.

Now suppose that $H\cup G_3 \neq E$ and consider the length $\ell+1$ bichain
\[\F'\vert \G'\coloneqq \{F_1\vert E \leq H\vert K \leq F_2\vert G_2 \leq F_3\vert G_3 \leq \ldots \leq F_{\ell-1}\vert G_{\ell-1} \leq E\vert G_\ell\}.\]
Because $\F'\vert\G'$ is a sub-bichain of $\H\vert \K$, $\F'\vert \G'$ is not a biflag.
We now check if $\F'\vert \G'$ has any repeated flats or coflats. The only case remaining to check is that $H\neq F_2$. Since $\F\vert \G$ has no gaps, we know that $F_2\cup G_3 = E$. Our current assumption is that $H\cup G_3 \neq E$, thus $F_2$ and $H$ are distinct.
Hence, $\F'\vert \G'$ is an element of $\mathbf{C}$. 
Once again, by our assumption, $\F'\vert G' \prec^* F\vert G$. This means that $\F'\vert \G'$ is not an element of $\Delta_{\prec\F\vert \G}$ and, because $\F'\vert \G'\subseteq \H\vert \K$, we have reached a contradiction.

We conclude that assuming $H\neq F_1$ and $K\neq G_2$ leads to a contradiction, which proves the second part of our lemma.
\qed

\noindent We use Lemma ~\ref{lem:NR_struct} to prove that the link of a non-repeating bichain $\F\vert \G$ in $\Delta_{\prec \F\vert \G}$ always contains a cone vertex. This will then let us use Lemmas ~\ref{lem:cone_collapse} and ~\ref{lem:collapse} to collapse $\Delta_{\prec\F\vert\G}$ onto $\dell_{\Delta_{\prec\F\vert\G}}\left(\F\vert\G\right)$.

\begin{lem}\label{lem:NR_cone}
Let $\F\vert \G$ be a non-repeating bichain written as
\[\F\vert \G = \{F_1\vert E \leq F_2\vert G_2 \leq \dots \leq F_{\ell-1}\vert G_{\ell-1}\leq E\vert G_\ell\}.\]
Then the link of $\F\vert \G$ in $\Delta_{\prec\F\vert \G}$ contains a cone vertex.
\end{lem}
\begin{proof}
Because $\F\vert \G$ is not a bichain, $F_1\cup G_2 =E$ and $F_1\vert G_2$ is a biflat. Also since $\F\vert \G$ is in $\mathbf{C}$, no flats or dual flats are repeated and $F_1\vert G_2\not\in \F\vert \G$. We now prove that $F_1\vert G_2$ is a cone vertex of $\lk_{\Delta_{\prec\F\vert \G}}(\F\vert \G)$. To accomplish this, we prove that if $\H\vert \K$ is a bichain of $\Delta_{\prec\F\vert \G}$ and $\F\vert \G\subseteq \H\vert \K$ but $F_1\vert G_2$ isn't a biflat of $\H\vert \K$, then $\H\vert \K \cup \{F_1\vert G_2\}$ is a bichain in $\Delta_{\prec\F\vert \G}$.

Suppose that $\H\vert \K$ is a bichain in $\Delta_{\prec\F\vert \G}$ containing $\F\vert \G$ but not $F_1\vert G_2$. As a consequence of Lemma ~\ref{lem:NR_struct}, $\H\vert \K$ is forced to have the form
\[\dots \leq F_1\vert E \leq F_1\vert K_1\leq F_1\vert K_2\leq \dots \leq F_1\vert K_\alpha \leq H_1\vert G_2 \leq H_2\vert G_2 \leq \dots \leq H_\beta\vert G_2 \leq F_2\vert G_2\leq \cdots \]
where 
\[E\supsetneq K_1 \supsetneq K_2\supsetneq \dots \supsetneq K_\alpha \supsetneq G_2 
\qquad \text{and} \qquad  F_1\subsetneq H_1 \subsetneq H_2 \subsetneq \dots \subsetneq H_\beta \subsetneq F_2\] are chains of dual flats and flats respectively. 
Thus we can see that $\H\vert \K\cup \{F_1\vert G_2\}$ is a bichain since it will have the form
\[\dots \leq F_1\vert E \leq F_1\vert K_1\leq F_1\vert K_2\leq \dots \leq F_1\vert K_\alpha \leq F_1\vert G_2 \leq H_1\vert G_2 \leq H_2\vert G_2 \leq \dots \leq H_\beta\vert G_2 \leq F_2\vert G_2\leq \cdots \]
It remains to check that $\H\vert \K\cup \{F_1\vert G_2\}$ is in $\Delta_{\prec\F\vert \G}$. That is, we need to check that no sub-bichain of $\H\vert \K\cup \{F_1\vert G_2\}$ is an element of $\mathbf{C}$ which precedes $\F\vert \G$ in $(\mathbf{C}, \prec^*)$. Suppose $\F'\vert \G' \subseteq \H\vert \K\cup \{F_1\vert G_2\}$ is a non-repeating bichain.  This bichain is either a sub-bichain of $\H\vert \K$ or it contains $\{F_1\vert G_2\}$.

If $\F'\vert G' \subseteq \H\vert\K$, then $\F\vert \G \prec^* \F'\vert \G'$ by virtue of $\H\vert\K$ being in $\Delta_{\prec\F\vert\G}$. 
Now suppose the biflat $F_1\vert G_2$ is in $\F'\vert \G'$. Let $F'\vert E$ be the unmixed biflat of $\F'\vert \G'$ with $E$ as its dual flat. By Lemma ~\ref{lem:NR_struct} we know that $F'\vert E \leq F_1\vert E$. Even further, this inequality has to be strict, since otherwise $\F'\vert \G'$ has a repeated flat, namely $F_1$. Thus, $F'\vert E \lneq F_1\vert E$, which means that $\F\vert \G \prec^* \F'\vert \G'$.
\end{proof}

\begin{example}
The matroid $M$, as defined in Example ~\ref{ex:k4-edge}, has $16$ bichains which are not biflags. Two of them are non-repeating,
\[\{125\vert E \leq E\vert 34\} \quad \text{ and } \quad \{345\vert E \leq E\vert 12\}. \]
Both bichains in $\mathbf{C}$ are incomparable under $\prec$. Lemma ~\ref{lem:NR_cone} implies that the link in $\Delta\left(\BF\right)$ of both non-repeating bichains will contain a cone point. We verify this by plotting these links. The left simplicial complex is the link of $\{125\vert E \leq E\vert 34\}$ in $\Delta\left(\BF\right)$ while the right simplicial complex is the link of $\{345\vert E\leq E\vert 12\}$ in $\Delta\left(\BF\right)$.

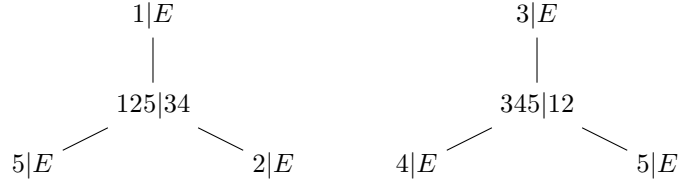
\begin{center}

\begin{tikzpicture} [scale=.4,auto=left]
  \node (12534) at (0,0) {$125\vert 34$};
  \node (5E) at (-4,-2) {$5\vert E$};
  \node (2E) at (4,-2) {$2\vert E$};
  \node (1E) at (0,3) {$1\vert E$};
  \foreach \from in {5E,2E,1E}
    \draw (\from) -- (12534);  
\end{tikzpicture}
\hspace{1cm}

\begin{tikzpicture} [scale=.4,auto=left]
  \node (12534) at (0,0) {$345\vert 12$};
  \node (5E) at (-4,-2) {$4\vert E$};
  \node (2E) at (4,-2) {$5\vert E$};
  \node (1E) at (0,3) {$3\vert E$};
  \foreach \from in {5E,2E,1E}
    \draw (\from) -- (12534);  
\end{tikzpicture}

\end{center}

\end{example}

\noindent We now have all the tools to prove Theorem ~\ref{thm:thm2}.

\begin{proof}
Suppose that our linear ordering $\preceq^*$ of the set of non-repeating bichains $\mathbf{C}= \{\F\vert \G, \H\vert \K, \ldots , \L\vert \M\}$ is one such that 
\[\F\vert \G \preceq^* \H\vert\K \preceq^*\ldots \preceq^* \L\vert\M. \]
By Lemma ~\ref{lem:NR_cone}, we know that the link of $\F\vert \G$ in $\Delta_{\prec \F\vert\G}$ has a cone vertex, the link of $\H\vert \K$ in $\Delta_{\prec \H\vert\K}$ has a cone vertex and so on. Lemmas ~\ref{lem:cone_collapse} and ~\ref{lem:collapse}, let us find a sequence of elementary collapses 
\[\Delta(\BF) = \Delta_{\prec\F\vert\G} \diagarrow \cdots \diagarrow \dell_{\Delta_{\prec\F\vert\G}}\left(\F\vert\G\right)\diagarrow \cdots \diagarrow \dell_{\Delta_{\prec \L\vert\M}}\left(\L\vert\M\right) = \Delta_{M,M^\perp}.\]
Therefore we have proven Theorem ~\ref{thm:thm2}.
\end{proof}

\section{Further Questions}\label{sec:future}

\subsection{``Global'' deformation retracts} \label{q:global}
Recall that the conormal fan $\Sigma_{M,M^\perp}$ is the restriction of the bipermutohedral fan $\Sigma_{E,E}$ to $\trop(M)\times \trop\left(M^\perp\right)$. The bipermutohedral fan has an associated simplicial complex $\Delta_{E,E}$ which will contain all conormal complexes as induced subcomplexes. This complex, which we call the \textit{bipermutohedral complex}, will have faces given by biflags of bisubsets. 

A \emph{bisubset} $S\vert T$ is a pair of non-empty subsets $S,T\subset E$ such that $S\cup T = E$ and $S\cap T \neq E$. Like in Definition ~\ref{def:biflat_complex}, we can define the simplicial complex $\Delta\left(\BS_{E,E}\right)$ whose faces are given by bichains of bisubsets. This simplicial complex contains the bipermutohedral complex as a subcomplex.

The fan structure of the product of Bergman fans $\Sigma_M\times\Sigma_{M^\perp}$ on $\trop(M)\times \trop\left(M^\perp\right)$ is the restriction of a larger fan, namely the product of permutohedral fans $\Sigma_E\times \Sigma_E$, see \cite{AK}. As in Lemma ~\ref{lem:unmixed_join}, the complex $\Delta_E*\Delta_E$ associated to $\Sigma_E\times \Sigma_E$ will be isomorphic to a subcomplex of $\Delta\left(\BS_{E,E}\right)$. 

We can thus recreate the diagram at the end of Section ~\ref{sec:conormal_back} for these larger ``global'' objects:
\[\begin{tikzcd}
	& {N_E\times N_E} \\
	{\Sigma_{E,E}} && {\Sigma_E\times \Sigma_{E}} \\
	{\Delta_{E,E}} & {\Delta(\BS_{E,E})} & {\Delta_E*\Delta_{E}}
	\arrow[squiggly, from=2-1, to=3-1]
	\arrow["\varphi", hook, from=3-1, to=3-2]
	\arrow[squiggly, from=2-3, to=3-3]
	\arrow["\psi"', hook', from=3-3, to=3-2]
	\arrow[squiggly, from=1-2, to=2-1]
	\arrow[squiggly, from=1-2, to=2-3]
\end{tikzcd}\]

It is unknown to the authors if the maps $\varphi$ and $\psi$ induce homotopy equivalences. Unlike the Bergman case, where $\Delta_E$ is the Bergman complex of $U_{n,n}$, none of these ``global'' complexes are related to a single matroid. Thus the proofs we give in Sections ~\ref{sec:unmixed} and ~\ref{sec:conormalcomplex} do not immediately carry over.

\begin{question}
Are $\Delta_{E,E}$ and $\Delta_E*\Delta_E$ homotopic to $\Delta\left(\BS_{E,E}\right)$? Can this be shown using combinatorial methods?
\end{question}

\subsection{$\Aut(M)$-equivariant homotopy equivalences}
For simplicial complexes and posets equipped with a $G$-action there exists a notion of $G$-equivariant homotopy equivalence. See \cite{TW} for the relevant definitions. As the simplicial complexes we construct only rely on matroidal information, they are all equipped with an $\Aut(M)$ action. It will be interesting to see if our proof methods can be upgraded to produce $\Aut(M)$-equivariant homotopy equivalences. 

\begin{question}
Is the biflats complex $\Aut(M)$-equivariantly homotopy equivalent to the conormal complex and the join of the Bergman complexes of the matroid and its dual? 
\end{question}

\printbibliography
\end{document}